\documentclass[11pt]{amsart}
\usepackage{amsmath}
\usepackage{amssymb}
\usepackage{amscd}

\def\NZQ{\mathbb}               
\def\NN{{\NZQ N}}
\def\QQ{{\NZQ Q}}
\def\ZZ{{\NZQ Z}}
\def\RR{{\NZQ R}}

\def\PP{{\NZQ P}}

\newtheorem{Theorem}{Theorem}[section]
\newtheorem{Lemma}[Theorem]{Lemma}
\newtheorem{Corollary}[Theorem]{Corollary}
\newtheorem{Proposition}[Theorem]{Proposition}
\newtheorem{Remark}[Theorem]{Remark}

\newtheorem{Example}[Theorem]{Example}

\let\epsilon\varepsilon
\let\phi=\varphi
\let\kappa=\varkappa

\begin{document}

\title{Analytic spread of filtrations on two dimensional normal local rings}
\author{Steven Dale Cutkosky}
\thanks{Partially supported by NSF grant DMS-2054394}

\address{Steven Dale Cutkosky, Department of Mathematics,
University of Missouri, Columbia, MO 65211, USA}
\email{cutkoskys@missouri.edu}


\begin{abstract} 
In this paper we prove that a classical theorem by McAdam about the  analytic spread of an ideal in a Noetherian local ring continues to be true for divisorial filtrations on a two dimensional normal excellent local ring $R$, and that the Hilbert polynomial of the fiber cone of a divisorial filtration on $R$ has a  Hilbert function which is  the sum of a linear polynomial and a bounded function. We prove these theorems by first studying asymptotic properties of divisors on a  resolution of singularities of the spectrum of $R$. The filtration of the symbolic powers of an ideal is an example of a divisorial filtration. 
Divisorial filtrations are often not Noetherian, giving a significant difference in the classical case of filtrations of powers of ideals and divisorial filtrations. 
\end{abstract}

\maketitle \section{Introduction} Divisorial filtrations on two dimensional normal excellent local rings have excellent properties, as we show in this article.

\subsection{Filtrations of powers of ideals and Analytic Spread} In this subsection we give an outline of how    the classical theory of the analytic spread of an ideal admits a simple  geometric interpretation  in the case of  an ideal in a normal excellent local ring. The generalization of analytic spread to divisorial filtrations can then be seen as a natural extension of this theory. 

Expositions of the theory of complete ideals, integral closure of ideals and their relation to valuation ideals, Rees valuations, analytic spread and birational morphisms can be found, from different perspectives,  in  \cite{ZS2}, \cite{HS}, \cite{Li2} and \cite{Li3}. The book \cite{HS} and the article \cite{Li3} contain references to  original work in this subject. Concepts in this introduction which are not defined in this section or in these references can be found in Section \ref{SecRes} of this paper. A survey of recent work on symbolic algebras is given in \cite{DDGHN}. A different notion of analytic spread for families of ideals is given in \cite{DM}.
A recent paper exploring ideal theory in two dimensional normal local domains using geometric methods is \cite{OWY}.

Let $R$ be a normal excellent  local ring with maximal ideal $m_R$ and $I$ be an ideal in $R$. Let $\pi:X\rightarrow \mbox{Spec}(R)$ be projective and birational (so that $\pi$ is the blow up of an ideal) and such that $X$ is normal and $I\mathcal O_X$ is an invertible sheaf. Let $I\mathcal O_X=\mathcal O_X(-D)$ where $D$ is an effective and anti-nef divisor (the intersection product $(D\cdot E)\le 0$ for all exceptional curves $E$ of $X$). Then $\Gamma(X,\mathcal O_X(-nD))=\overline{I^n}$, the integral closure of $I^n$,  for all $n\in \NN$. Write $D=a_1F_1+\cdots+a_sF_s$ where the $F_i$ are prime divisors.
The local rings $\mathcal O_{X,F_i}$ are discrete (rank 1) valuation rings. Let $\nu_{F_i}$ be the associated valuations. 
 We have that the integral closure of $I^n$ is
$$
\overline{I^n}=\Gamma(X,\mathcal O_X(-nD))=I(\nu_{F_1})_{na_1}\cap \cdots\cap I(\nu_{F_s})_{na_s}
$$
where 
$$
I(\nu_{F_i})_b=\{f\in R\mid \nu_{F_i}(f)\ge b\}
$$
are the valuation ideals in $R$ associated to $\nu_{F_i}$. The center of $\nu_{F_i}$ on $R$ is the prime ideal $I(\nu_{F_i})_1$.
The Rees valuations of $I$ are those $\nu_{F_j}$ such that
$\overline{I^n}\neq \cap_{i\ne j}I(\nu_{F_i})_{na_i}$.  Let $Y$ be the normalization of the blow up $B(I)$ of $I$ and let $I\mathcal O_Y=\mathcal O_Y(-B)$. Then $Y\rightarrow \mbox{Spec}(R)$ is projective (since $R$ is universally Nagata). The divisor  $-B$ is ample on $Y$ and so the Rees valuations  of $I$ are exactly the  prime components of $B$. 
By the universal property of blowing up, $\pi$ factors through $B(I)$ and since $X$ is normal,
$\pi$ factors through $Y$. Let $\phi:X\rightarrow Y$ be the induced morphism. Let $F$ be a prime component of $D$, with associated valuation $\nu_F$. Then $\nu_F$ is a Rees valuation of $I$ if and only if $\phi$ does not contract $F$, in which case   $\phi(F)=E$ is a prime component of $B$ and we have that $\mathcal O_{X,F}=\mathcal O_{Y,E}$.  

In the case that $\dim R=2$,  the prime divisor $F$ is contracted by $\phi$ if and only if $F$ is exceptional ($\pi(F)=m_R$) and $(D\cdot F)=0$. Thus the Rees valuations of $I$ are precisely the valuations associated to prime divisors $F$ of $X$ such that either $\nu_F$ has center a height one prime of $R$ or $F$ is exceptional for $\pi$ (the center of $\nu_F$ on $R$ is $m_R$) and  $(D\cdot F)<0$.

Let us return to not having any restrictions on the dimension of $R$. We have an associated graded ring
$R[It]=\sum_{n\ge 0}I^nt^n$ (the Rees algebra of $I$). The integral closure of $R[It]$ in $R[t]$ is the graded algebra
$\overline{R[It]}=\sum_{n\ge 0}\overline{I^n}t^n$, which is a finite extension of $R[It]$ (since $R$ is universally Nagata). 
The blow up of $I$ is $B(I)=\mbox{Proj}(R[It])$ and $Y=\mbox{Proj}(\overline{R[It]})$ is the normalization of the blow up of $I$, which was introduced earlier.  Let $\psi:B(I)\rightarrow \mbox{Spec}(R)$ be the projection.

The blowup $B(I)$ has the important subschemes 
$$
\psi^{-1}(V(I))=\mbox{Proj}(\mbox{gr}_I(R))\mbox{ and }\psi^{-1}(m_R)=\mbox{Proj}(R[It]/m_RR[It]).
$$
 The $R$-algebra $\mbox{gr}_I(R)=\sum_{n\ge0}I^n/I^{n+1}t^n$  is the associated graded ring of $I$ and the $R$-algebra $R[It]/m_RR[It]$ is  the fiber cone of $I$.

Since $\mbox{Proj}(R[It])\rightarrow \mbox{Spec}(R)$ and $\mbox{Proj}(\overline{R[It]})\rightarrow \mbox{Spec}(R)$
are birational, the dimensions of $\mbox{Proj}(R[It])$ and $\mbox{Proj}(\overline{R[It]})$ are the same as the dimension of $R$.
Further, since $\mbox{Proj}(\mbox{gr}_I(R))$ is a Cartier divisor on $\mbox{Proj}(R[It])$, we have that $\dim (\mbox{Proj}(\mbox{gr}_I(R))=\dim R-1$.  Now, since $I\subset m_R$, we have that $\mbox{Proj}(R[It]/m_RR[It])$ is a subscheme of 
$\mbox{Proj}(\mbox{gr}_I(R))$, so we have $\dim(\mbox{Proj}(R[It]/m_RR[It]))\le \dim R-1$.

Let $\psi_0:\mbox{Proj}({\rm gr}_I(R))\rightarrow \mbox{Spec}(R/I)$ be the projective morphism induced by $\psi$. Let $P$ be a minimum prime of $I$. Then $\dim \psi_0^{-1}(P)=\dim R_P-1$ since $I_p$ is primary for the maximal ideal of $R_P$. We have that  $\dim \psi^{-1}(m_R)=\dim\psi_0^{-1}(m_R)\ge \dim \psi_0^{-1}(P)$ by upper semi-continuity of fiber dimension (\cite[Corollary IV.13.1.5]{EGA28}). Thus 
$$
\mbox{ht}(I)\le \dim \psi^{-1}(m_R)+1.
$$

The analytic spread of $I$ is defined to be 
$$
\ell(I)=\dim R[It]/m_RR[It].
$$ 

Since the dimension of the Proj of a graded ring is one less than the dimension of the ring, we have established in our case of normal excellent local rings the following theorems.

\begin{Theorem}\label{TheoremC1}(\cite[Proposition 5.1.6 and Corollary 8.3.9]{HS}) Let $R$ be a Noetherian local ring and $I$ be an ideal in $R$. Then 
$$
{\rm ht}(I)\le \ell(I)\le \dim {\rm gr}_I(R)=\dim R.
$$
\end{Theorem}

\begin{Theorem}\label{TheoremC2}(\cite[Proposition 5.4.8]{HS}) Let $R$ be a Noetherian formally equidimensional local ring and let $I$ be an ideal in $R$. For every minimal prime ideal $P$ of ${\rm gr}_I(R)$, $\dim({gr}_I(R)/P)=\dim R$.
\end{Theorem}


We return to the case that $R$ is a normal excellent local ring of arbitrary dimension.
We have that $\ell(I)=\dim R$ if and only if $\dim \psi^{-1}(m_R)=\dim R-1$.
Since 
$$
Y=\mbox{Proj}(\overline{R[It]})\rightarrow B(I)=\mbox{Proj}(R[It])
$$
 is finite, $\dim \psi^{-1}(m_R)=\dim R-1$ if and only if there exists a prime divisor $E$ on $Y$ which contracts to $m_R$; that is, the center of $\nu_E$  on $R$ is $m_R$.  Writing 
 $$
 I\mathcal O_Y=\mathcal O_Y(-b_1F_1-\cdots - b_sF_s)
 $$
  where $F_i$ are prime divisors and $b_i>0$, we have that 
 $$
 \overline{I^n}=I(\nu_{F_1})_{nb_1}\cap \cdots \cap I(\nu_{F_s})_{nb_s}
 $$
 where $\nu_{F_i}$ is the discrete rank 1 valuation associated to the valuation ring $\mathcal O_{Y,F_i}$.  Since $-b_1F_1-\cdots-b_sF_s$ is ample on $Y$,
 we have that $\overline{I^n}\ne\cap_{i\ne j} I(\nu_{F_i})$ for all $j$ and $n\gg 0$ (so that $\nu_{F_1},\ldots,\nu_{F_s}$ are the Rees valuations of $I$). Thus $\dim \psi^{-1}(m_R)=\dim R-1$ holds if and only if $m_R\in \mbox{Ass}(R/\overline{I^n})$ for some $n$.  
   
 We have established the following theorem in our case of normal excellent local rings.

\begin{Theorem}\label{TheoremC4}(\cite{McA}, \cite[Theorem 5.4.6]{HS}) Let $R$ be a formally equidimensional local ring and $I$ be an ideal in $R$. Then $m_R\in \mbox{Ass}(R/\overline{I^n})$ for some $n$ if and only if $\ell(I)=\dim(R)$.
\end{Theorem}

The assumption of being formally equidimensional is not required for the if direction of Theorem \ref{TheoremC4} (this is Burch's theorem, \cite{Bu}, \cite[Proposition 5.4.7]{HS}).

Let $k=R/m_R$. Since $R[It]/m_RR[It]$ is a standard graded ring over $k$ (finitely generated in degree 1) it has a Hilbert polynomial P(n) which has degree $d=\ell(I)-1$; there exists a positive integer $n_0$ such that
\begin{equation}\label{eqI2}
\dim_k I^n/m_RI^n=  P(n)\mbox{ for }n\ge n_0.
\end{equation}
As $\overline{R[It]}/m_R\overline{R[It]}$ is a finitely generated graded ring over $k$, there exists $e\in \ZZ_{>0}$ and polynomials $P_0,\ldots, P_{e-1}$ of degree $d=\ell(I)-1$ such that
\begin{equation}\label{eqC5}
\dim_k \overline{I^n}/m_R\overline{I^n}=  P_i(n)\mbox{ for $n\ge n_0$ where $i\equiv n\mbox{ mod }e$}.
\end{equation}

\subsection{Filtrations}
Let $\mathcal I=\{I_n\}$ be a filtration on a local ring $R$. The Rees algebra of the filtration is $R[\mathcal I]=\oplus_{n\ge 0}I_n$. Analogously to the case of ideals, we define the fiber cone of the filtration $\mathcal I$ to be $R[\mathcal I]/m_RR[\mathcal I]$ and the analytic spread of the filtration of $\mathcal  I$ to be
\begin{equation}\label{eqN10}
\ell(\mathcal I)=\dim R[\mathcal I]/m_RR[\mathcal I].
\end{equation}
We have that ${\rm ht}(I_n)={\rm ht}(I_1)$ for all $n$ (\cite[equation (7)]{CPS}) so it is natural to define ${\rm ht}(\mathcal I)={\rm ht}(I_1)$.

We always have (\cite[Lemma 3.6]{CPS}) that 
$$
\ell(\mathcal I)\le \dim R
$$
so the second inequality of Theorem \ref{TheoremC1} always holds. 
 However, the first inequality of Theorem \ref{TheoremC1},
$
\mbox{ht}(\mathcal I)\le \ell(\mathcal I),
$
fails spectacularly, even attaining the condition that $\ell(\mathcal I)=0$ (\cite[Example 1.2, Example 6.1 and Example 6.6]{CPS}). The last two of these examples are of symbolic algebras of space curves, which are  divisorial filtrations. We give a further example where the inequality fails in Example \ref{Ex1} of this paper. Example \ref{Ex1} is of a symbolic algebra of an intersection of height 1 prime ideals in a two dimensional excellent normal local ring.
In the case that $\mathcal I$ is a Noetherian filtration ($R[\mathcal I]$ is a finitely generated $R$-algebra),  the lower bound ${\rm ht}(\mathcal I)\le \ell(\mathcal I)$ always holds
 (\cite[Proposition  3.7]{CPS}), so that the  inequality of Theorem \ref{TheoremC1} for ideals continues to hold for Noetherian filtrations. 
 
 The condition that  a filtration has  analytic spread zero has a simple ideal theoretic interpretation (\cite[Lemma 3.8]{CPS}).
 Suppose that  $\mathcal I=\{I_n\}$ is a filtration in a local ring $R$. Then the analytic spread
$\ell(\mathcal I)=0$ if and only if
$$
\mbox{For all $n>0$ and $f\in I_n$, there exists $m>0$ such that $f^m\in m_RI_{mn}$.}
$$

\subsection{Divisorial Filtrations}
 Let $R$ be a  local domain of dimension $d$ with quotient field $K$.  Let $\nu$ be a discrete valuation of $K$ with valuation ring $V_{\nu}$ and maximal ideal $m_{\nu}$.  Suppose that $R\subset V_{\nu}$. Then for $n\in \NN$, define valuation ideals
$$
I(\nu)_n=\{f\in R\mid \nu(f)\ge n\}=m_{\nu}^n\cap R.
$$
 
  A divisorial valuation of $R$ (\cite[Definition 9.3.1]{HS}) is a valuation $\nu$ of $K$ such that if $V_{\nu}$ is the valuation ring of $\nu$ with maximal ideal $m_{\nu}$, then $R\subset V_{\nu}$ and if $p=m_{\nu}\cap R$ then $\mbox{trdeg}_{\kappa(p)}\kappa(\nu)={\rm ht}(p)-1$, where $\kappa(p)$ is the residue field of $R_{p}$ and $\kappa(\nu)$ is the residue field of $V_{\nu}$.   If $\nu$ is divisorial valuation such that $m_R= m_{\nu}\cap R$, then $\nu$ is called an $m_R$-valuation.
   
 By \cite[Theorem 9.3.2]{HS}, the valuation ring of every divisorial valuation $\nu$ is Noetherian, hence is a  discrete valuation. 
 Suppose that  $R$ is an excellent local domain. Then a valuation $\nu$ of the quotient field $K$ of $R$ which is nonnegative on $R$ is a divisorial valuation of $R$ if and only if the valuation ring $V_{\nu}$  of $\nu$ is essentially of finite type over $R$ (\cite[Lemma 5.1]{CPS1}).
 
 In general, the filtration $\mathcal I(\nu)=\{I(\nu)_n\}$ is not Noetherian; that is, the graded $R$-algebra $\sum_{n\ge 0}I(\nu)_nt^n$ is not a finitely generated $R$-algebra. In a two dimensional normal local ring $R$, the condition that the filtration of valuation ideals $\mathcal I(\nu)$ is Noetherian for all $m_R$-valuations $\nu$ dominating $R$  is the condition (N) of Muhly and Sakuma \cite{MS}. It is proven in \cite{C2} that a complete normal local ring of dimension two satisfies condition (N) if and only if its divisor class group is a torsion group.





An integral  divisorial filtration of $R$ (which we will refer to as a divisorial filtration in this paper) is a filtration $\mathcal I=\{I_m\}$ such that  there exist divisorial valuations $\nu_1,\ldots,\nu_s$ and $a_1,\ldots,a_s\in \ZZ_{\ge 0}$ such that for all $m\in \NN$,
$$
I_m=I(\nu_1)_{ma_1}\cap\cdots\cap I(\nu_s)_{ma_s}.
$$




$\mathcal I$ is called an $\RR$-divisorial filtration if $a_1,\ldots,a_s\in \RR_{>0}$ and $\mathcal I$ is called a $\QQ$-divisorial filtration if $a_1,\ldots,a_s\in \QQ$. If $a_i\in \RR_{>0}$, then
$$
I(\nu_i)_{na_i}:=\{f\in R\mid \nu_i(f)\ge na_i\}=I(\nu_i)_{\lceil na_i\rceil},
$$
where $\lceil x\rceil$ is the round up of a real number.

Given an ideal $I$ in $R$, the filtration $\{\overline{I^n}\}$ is an example of a divisorial filtration of $R$. The filtration 
$\{\overline{I^n}\}$ is Noetherian if $R$ is universally Nagata.

It is shown in \cite[Theorem 4.5 ]{CPS} that 
 the ``if" statement of  Theorem \ref{TheoremC4} is true for divisorial filtrations of a local domain $R$.
 
 \begin{Theorem}(\cite[Theorem 4.5]{CPS})\label{TheoremN2} Suppose that $R$ is a local domain and $\mathcal I=\{I_n\}$ is  a divisorial filtration on $R$ such that $\ell(I)=\dim R$. Then $m_R\in \mbox{Ass}(R/\overline{I^n})$ for infinitely many $n$.
 \end{Theorem} 
 
 An interesting question is if the converse of Theorem \ref{TheoremC4} is also true for divisorial filtrations of a local ring $R$.
 We prove this for two dimensional excellent normal local rings in this paper (Theorem \ref{Cor2}, also stated in Theorem \ref{ThmI1} of this introduction).
 
 \subsection{Divisorial filtrations on normal excellent local rings} Let $R$ be a normal excellent local ring. Let $\mathcal I=\{I_m\}$ where 
 $$
I_m=I(\nu_1)_{ma_1}\cap\cdots\cap I(\nu_s)_{ma_s}.
$$ 
 for some divisorial valuations $\nu_1,\ldots,\nu_s$ on $R$ be an $\RR$-divisorial filtration on a normal excellent local ring $R$, with $a_1,\ldots, a_s\in \RR_{>0}$. Then there exists a projective birational morphism $\phi:X\rightarrow \mbox{Spec}(R)$ such that there exist prime divisors $F_1,
 \ldots, F_s$ on $X$ such that $V_{\nu_i}=\mathcal O_{X,F_i}$ for  $1\le i\le s$. Let $D=a_1F_1+\cdots+a_sF_s$, an effective $\RR$-divisor. Define $\lceil D\rceil=\lceil a_1\rceil F_1+\cdots+\lceil a_s\rceil F_s$, an integral divisor. 
 We have coherent sheaves $\mathcal O_X(-\lceil n D\rceil)$ on $X$ such that 
 \begin{equation}\label{N1}
 \Gamma(X,\mathcal O_X(-\lceil nD\rceil ))=I_n
 \end{equation}
 for $n\in \NN$. If $X$ is nonsingular then $\mathcal O_X(-\lceil nD\rceil)$ is invertible. The formula (\ref{N1}) is independent of choice of $X$. Further, even on a particular $X$, there are generally many different choices of effective $\RR$-divisors $G$ on $X$ such that $\Gamma(X,\mathcal O_X(-\lceil nG\rceil))=I_n$ for all $n\in \NN$. Any choice of a divisor $G$ on such an $X$ for which the formula $\Gamma(X,\mathcal O_X(-\lceil nG\rceil))=I_n$ for all $n\in \NN$ holds will be called a representation of the filtration  $\mathcal I$. 
 
 Given an  $\RR$-divisor $D=a_1F_1+\cdots+a_sF_s$ on $X$ we have a divisorial filtration
 $\mathcal I(D)=\{I(D)_n\}$ where 
 $$
 I(D)_n=\Gamma(X,\mathcal O_X(-\lceil nD\rceil ))=I(\nu_1)_{\lceil na_1\rceil}\cap\cdots\cap I(\nu_s)_{\lceil na_s\rceil} 
 =I(\nu_1)_{ma_1}\cap\cdots\cap I(\nu_s)_{ma_s}.
 $$ 
 We write $R[D]=R[\mathcal I(D)]$.

 \subsection{Summary of principal results in this paper} Let $R$ be an excellent two dimensional normal excellent local ring with maximal ideal $m_R$.
 
 All possible analytic spreads $\ell(\mathcal I(D))=0,1,2$ can occur for $\QQ$-divisors $D$ on $R$. An example where $\ell(\mathcal I(D))=0<{\rm ht}(\mathcal I(D))=1$ is given in Example \ref{Ex1}.
 This example is of a symbolic filtration $\mathcal I(D)=\{Q_1^{(n)}\cap Q_2^{(n)}\cap Q_3^{(n)}\}$ where $Q_1,Q_2,Q_3$ are height one prime ideals in a two dimensional normal excellent local ring $R$. 
 In contrast, since the filtration $\mathcal I(D)$ is not Noetherian, we have (by \cite[Corollary 1.9]{CPS}) that for every $a\in \ZZ_{>0}$, the analytic spread of the ideal $Q_1^{(a)}\cap Q_2^{(a)}\cap Q_3^{(a)}$  is $\ell(Q_1^{(a)}\cap Q_2^{(a)}\cap Q_3^{(a)})=2$, the largest possible.

  We prove that the conclusions of Theorem \ref{TheoremC4} hold for $\QQ$-divisorial filtrations on $R$ in Theorem \ref{Cor2}.
   
\begin{Theorem}\label{ThmI1}(Theorem \ref{Cor2}) Let $R$ be a two dimensional normal excellent local ring. The following are equivalent for a $\QQ$-divisorial filtration $\mathcal I(D)$ on $R$.
\begin{enumerate}
\item[1)]
The analytic spread $\ell(\mathcal I(D))=\dim R[D]/m_RR[D]=2$.
\item[2)] $m_R\in \mbox{Ass}(R/I(nD))$ for some $n$.
\item[3)] There exists $n_0\in \ZZ_{>0}$ such that $m_R\in \mbox{Ass}(R/I(nD))$ for all $n\ge n_0$.
\end{enumerate}
\end{Theorem}
 
 We generalize the formula on Hilbert functions of filtrations of powers of ideals in (\ref{eqI2}) and (\ref{eqC5}) to $\QQ$-divisorial filtrations on $R$ in Theorem \ref{HilbThm}.

\begin{Theorem}(Theorem \ref{HilbThm}) Suppose that $R$ is a two dimensional normal excellent local ring and $\mathcal I(D)$ is a $\QQ$-divisorial filtration on $R$.
Then there exist a nonnegative rational number $\alpha$ and a bounded function $\sigma:\NN\rightarrow\QQ$ such that 
$$
\ell_R(I(nD)/m_RI(nD))=\ell_R((R[D]/m_RR[D])_n)=n\alpha+\sigma(n)
$$
for $n\in \NN$.
The constant $\alpha$ is positive if and only if $\dim(R[D]/m_RR[D])=2$.
\end{Theorem}

 It is unlikely that the function $\sigma(n)$ will always be eventually periodic. 
 It is shown in \cite[Theorem 9]{CS} that if $D$ has exceptional support then
 the Hilbert function of ${\rm gr}_{\mathcal I}(R)=\sum_{n\ge 0}I(nD)/I((n+1)D)t^n$ has an expression
 $$
 \ell_R(I(nD)/I((n+1)D))=n\beta+\tau(n)
 $$
 where $\beta\in \QQ$ and $\tau(n)$ is a bounded function. If $R$ has equicharacteristic zero then it is shown in \cite[Theorem 9]{CS} that $\tau(n)$ is eventually periodic, and \cite[Example 5]{CS} gives an example where  $R$ has equicharacteristic $p>0$ and $\tau(n)$ is not  eventually periodic.
 
 Suppose that $A$ is an excellent normal local ring of dimension 3. Let $Z\rightarrow \mbox{Spec}(A)$ be  a resolution of singularities and $D$ be an effective divisor on $Z$, all of whose components contract to the maximal ideal $m_A$.  Then the Hilbert polynomial
$ h(n)=\ell_A(I(nD)/I((n+1)D))$  may be far from being polynomial like. The examples (\cite[Example 6]{CS} and \cite[Theorem 1.4]{C3}) have the property that
$$
\lim_{n\rightarrow \infty}\frac{h(n)}{n^2}
$$
is an irrational number. These examples are in three dimensional equicharacteristic rings $A$ of any characteristic. The reason for this irrational behavior in dimension three is because of the lack of existence of Zariski decompositions in dimension three.

 We now give an outline of the proof of Theorem \ref{Cor2}.
Let $\pi:X\rightarrow \mbox{Spec}(R)$ be a resolution of singularities such that $D$ is represented on $X$. Let $E_1,\ldots,E_r$ be the prime exceptional divisors of $\pi$. An $\RR$-divisor $\Delta$ on $X$ is anti-nef if $(E\cdot \Delta)\le 0$ for all prime exceptional divisors $E$ on $X$.
Since $X$ has dimension two, $D$ has a Zariski decomposition, $\Delta=D+B$ where $\Delta$ is an anti-nef divisor and $B$  
is an effective divisor with exceptional support such that
$$
I(nD)=\Gamma(X,\mathcal O_X(-\lceil nD\rceil))=\Gamma(X,\mathcal O_X(-\lceil n\Delta\rceil))=I(n\Delta)
$$
for all $n\in \NN$.  This decomposition does not exist in higher dimensions, even after blowing up (\cite{C1},  \cite[Section IV.2.10]{Nak}, \cite[Section 2.3]{LA}).

\begin{Proposition}(Corollary \ref{Cor3}) Suppose that $\Delta$ is an effective  anti-nef $\QQ$-divisor on $X$. Then the following are equivalent.
\begin{enumerate}
\item[1)] There exists $n$ such that $m_R\in \mbox{Ass}(R/I(n\Delta))$.
\item[2)] There exists $n_0$ such that $m_R\in \mbox{Ass}(R/I(n\Delta))$ for all $n\ge n_0$.
\item[3)] There exists $j$ such that $E_j$ is exceptional and $(\Delta\cdot E_j)<0$.
\end{enumerate}
\end{Proposition}

 Let $E_j$ be an exceptional divisor of $\pi$ and 
 $$
 P_j=\bigoplus_{n\ge 0}\Gamma(X,\mathcal O_X(-\lceil n\Delta\rceil-E_j))
 $$
  for $1\le j\le r$. $P_j$ is a prime ideal   in
 $R[\Delta]=R[D]$. In Proposition \ref{Prop80} it is shown that 
 $$
 \sqrt{m_RR[\Delta]}=\cap_{i=1}^rP_i.
 $$
  The following proposition computes the dimension of $R[\Delta]/P_j$ in terms of the intersection theory of $X$.

\begin{Proposition}(Proposition \ref{Prop3}) Suppose that $\Delta$ is an effective anti-nef $\QQ$-divisor on $X$ and $E_j$ is a prime  exceptional divisor for $\pi:X\rightarrow\mbox{Spec}(R)$. Then 
\begin{enumerate}
\item[1)] $\dim R[\Delta]/P_j=2$ if  $(\Delta\cdot E_j)<0$.
\item[2)] $\dim R[\Delta]/P_j\le 1$ if $(\Delta\cdot E_j)=0$.
\end{enumerate}
\end{Proposition}

 Since $\sqrt{m_RR[\Delta]}=\cap_{i=1}^rP_i$, we deduce Theorem \ref{Cor2} from Propositions \ref{Cor3} and \ref{Prop3}. 
 
 The theory of Zariski decomposition was created and developed by Zariski in \cite{Z} for projective surfaces over an algebraically closed field. In Section \ref{SecZD}, we give the relative version of this theory, over a two dimensional excellent normal local ring, and in Section \ref{SecNef}, we extend some results in \cite{Z} for numerically effective divisors on a nonsingular projective surface to our situation of a resolution of singularities of a two dimensional normal excellent local ring. 
 We prove the main results of this paper on asymptotic properties of divisors on a resolution of singularities of  a two dimensional normal excellent local ring in Section  \ref{SecAS}. We prove Theorem \ref{Cor2} in Section \ref{SecAS2} and Theorem \ref{HilbThm} in Section \ref{SecHilb}.

 \subsection{Notation}
 We will denote the nonnegative integers by $\NN$ and the positive integers by $\ZZ_{>0}$,   the set of nonnegative rational numbers  by $\QQ_{\ge 0}$  and the positive rational numbers by $\QQ_{>0}$. 
We will denote the set of nonnegative real numbers by $\RR_{\ge0}$ and the positive real numbers by $\RR_{>0}$. 
If $x\in \RR$, then $\lceil x\rceil$ is the smallest integer which is greater than or equal to $x$.

The maximal ideal of a local ring $R$ will be denoted by $m_R$. 
We will denote the length of an $R$-module $M$ by $\ell_R(M)$.
 \cite[Scholie IV.7.8.3]{EGAIV} gives a list of good properties of excellent local rings which we will assume.


\section{Divisors on a resolution of singularities of a two dim. local ring}\label{SecRes}
Throughout this paper 
$R$ will be a two dimensional excellent normal local ring with quotient field $K$, maximal ideal $m_R$ and residue field $k=R/m_R$.

From this section through Section \ref{SecAS},
$\pi:X\rightarrow \mbox{Spec}(R)$ will  be a resolution of singularities  such that  $\pi$ is projective and
all exceptional prime divisors of $\pi$ are nonsingular. Such a resolution of singularities exists by \cite{Li} or \cite{CJS}. Let $E_1,\ldots,E_r$ be the exceptional prime divisors for $\pi$. A divisor is exceptional if all its prime components map to $m_R$ by $\pi$. We will further assume that $\pi$ is not an isomorphism.

\begin{Remark}\label{RemarkR3}  Suppose that $\mathcal F$ is a coherent sheaf on $X$. Then $H^0(X,\mathcal F)$ is a finitely generated $R$-module, $H^1(X,\mathcal F)$ is an $R$ module of finite length and $H^2(X,\mathcal F)=0$.
\end{Remark}

\begin{proof}  By  \cite[Theorem III.5.2]{H}, $H^0(X,\mathcal F)$ is a finitely generated $R$-module.
By \cite[Theorem III.5.2 and Corollary III.11.2]{H},  $H^1(X,\mathcal F)$ is an $R$ module of finite length and
by \cite[Corollary III.11.2]{H}, $H^2(X,\mathcal F)=0$ since $\dim \pi^{-1}(m_R)=1$.
\end{proof}

An element of the free abelian group $\mbox{Div}(X)$ on the prime divisors of $X$ is called a divisor. 
Elements of $\mbox{Div}(X)\otimes \QQ$ are called $\QQ$-divisors and elements of $\mbox{Div}(X)\otimes\RR$ are called $\RR$-divisors. We will sometimes refer to a divisor as an integral divisor if we want to emphasize this fact. 
If $D_1$ and $D_2$ are $\RR$-divisors then write $D_2\ge D_1$ if $D_2-D_1$ is an effective divisor. The degree $\deg(\mathcal L)$ for $\mathcal L$ an invertible sheaf on a projective curve  is defined in Section \ref{SecRR}.

We use the intersection theory on $X$ developed in \cite[Sections 12 and 13]{Li2}.  
The intersection theory on $X$ is determined by the formula $(D\cdot E)=\deg(\mathcal O_X(D)\otimes\mathcal O_E)$ if $D$ is a divisor on $X$ and   $E$ is a prime exceptional  divisor on $X$.

An $\RR$-divisor $D$ is numerically effective (nef) if $(E\cdot D)\ge 0$ for all prime exceptional divisors $E$ of $X$. 
An $\RR$-divisor $D$ on $X$ is anti-effective or anti-nef if $-D$ is respectively effective or nef. A $\QQ$-divisor $D$ is anti-ample if $-D$ is ample and an (integral) divisor $D$ is anti-very ample if $-D$ is very ample.

Let $F$ be a prime divisor on $X$. Then $\mathcal O_{X,F}$ is a (rank 1) discrete valuation ring. Let $\nu_F$ be the associated valuation. For $0\ne f\in K$ the divisor of $f$ on $X$ is $(f)=\sum \nu_F(f)F$ where the sum is over all the prime divisors $F$ of $X$. Two divisors $D_1$ and $D_2$ are linearly equivalent, written $D_1\sim D_2$ if there exists $f\in K$ such that $(f)=D_2-D_1$. Two divisors $D_1$ and $D_2$ which are linearly equivalent are also numerically equivalent; that is, $(E\cdot D_2)=(E\cdot D_1)$ for all prime exceptional divisors $E$ of $\pi$.

Let $D=\sum b_iF_i$ be an integral divisor on $X$. There is an associated invertible sheaf $\mathcal O_X(D)$ on $X$ which is determined by  the property that  if $U$ is an affine open subset of $X$ and $h\in K$ is such that $h=0$ is a local equation of $D$ in $U$, then $\mathcal O_X(D)\mid U=\frac{1}{h}\mathcal O_U$.
Thus 
$$
\Gamma(X,\mathcal O_X(D))=\{f\in R\mid (f)+D\ge 0\}. 
$$

 Since  $R$ is a subset of $\Gamma(X,\mathcal O_X)$ in $K$ and $R$ is normal we have that
$\Gamma(X,\mathcal O_X)=R$ by Remark \ref{RemarkR3},
and so if $D$ is an effective divisor then $\Gamma(X,\mathcal O_X(-D))$ is an ideal in $R$.

Let $\lceil a\rceil$ denote the smallest integer that is greater than or equal to a real number $a$. If $D=\sum_{i=1}^s a_iF_i$ with $a_i\in \RR$ is an $\RR$-divisor, let $\lceil F\rceil = \sum \lceil a_i\rceil F_i$.
 
 Let $F$ be a prime divisor on $X$. For $\alpha\in \RR_{\ge 0}$ define valuation ideals in $R$ by
 $$
 I(\nu_F)_{\alpha}=\{f\in R\mid \nu_F(f)\ge\alpha\}.
 $$
 We necessarily have that $I(\nu_F)_{\alpha}=I(\nu_F)_{\lceil\alpha\rceil}$.

For an effective $\RR$-divisor $D=a_1F_1+\cdots+a_sF_s$, where $F_1,\ldots,F_s$ are prime divisors on $X$ and $a_i\in \RR_{\ge 0}$, we have an associated ideal in $R$
$$
I(D):=I(\nu_{F_1})_{a_1}\cap \cdots \cap I(\nu_{F_s})_{a_s}=I(\nu_{F_1})_{\lceil a_1\rceil}\cap \cdots \cap I(\nu_{F_s})_{\lceil a_s\rceil}=\Gamma(X,\mathcal O_X(-\lceil D\rceil)).
$$

Let $D$ be a  divisor on $X$. Then $\Gamma(X,\mathcal O_X(D))\ne 0$.
The fixed component of $D$ is the largest effective  divisor $F$ on $X$ such that 
$$ 
\Gamma(X,\mathcal  O_X(D))=\Gamma(X,\mathcal O_X(D-F)).
$$

For $n\in \NN$, let $B_n$ be the fixed component of $nD$ and let
$$
M_i=\{n\in \NN\mid E_i\mbox{ is not a component of }B_n\}.
$$
$M_i$ is a numerical semigroup, so if $M_i$ is nonzero, there exists $h_i\in \ZZ_{>0}$ such that for $n\gg 0$, $n\in M_i$ if and only if $h_i$ divides $n$.

The global sections $\Gamma(X,\mathcal O_X(D))$ of $\mathcal O_X(D)$ generate $\mathcal O_X(D)$ at a point $q\in X$ if $\mathcal O_X(D)_q=\Gamma(X,\mathcal O_X(D))\mathcal O_{X,q}$. The points $q\in X$ where $\mathcal O_X(D)$ is generated by global sections are necessarily disjoint from the support of the fixed component of $D$.

\begin{Lemma}\label{Lemma19} Let $D$ be an effective  divisor  on $X$ and let $F$ be a prime divisor in the support of the fixed component of $-D$. Then the support of $F$ is exceptional.
\end{Lemma}

\begin{proof}  Write $D=\sum_{i=1}^r a_iF_i$ where the $F_i$ are distinct prime divisors on $X$ and $a_i\in \NN$. Suppose that $F_j$ is not exceptional for $\pi$. Let $q_j=\pi(F_j)$, a height one prime ideal in $R$. Since $\pi$ is an isomorphism over $\mbox{Spec}(R)\setminus m_R$, we have that 
$R_{q_j}=\mathcal O_{X,F_j}$, so
$$
\begin{array}{lll}
\mathcal O_X(-D)_{F_j}&=&(q_j^{a_j})_{q_j}=(I(\nu_j)_{a_j})_{q_j}=\Gamma(X,\mathcal O_X(-D))_{q_j}\\
&=&\Gamma(X,\mathcal O_X(-D ))\mathcal O_{X,F_j}.
\end{array}
$$
Thus $F_j$ is not in the support of $F$. 
\end{proof}

The intersection matrix of the exceptional curves of $\pi$ is the $r\times r$ matrix $\left((E_i\cdot E_j)\right)$ which is negative definite (\cite[Lemma 14.1]{Li2}).

\begin{Proposition}\label{PropAmp} Let $D$ be a $\QQ$-divisor on $X$. Then $D$ is ample if and only if $(D\cdot E)>0$ for all prime exceptional divisors $E$ on $X$.
\end{Proposition}

This is proved in \cite[Theorem 12.1]{Li2}. As commented in the proof of \cite[Theorem 12.1]{Li2}, the additional assumption there that $H^1(X,\mathcal O_X)=0$ is not necessary for this conclusion.

\begin{Lemma}\label{RemarkAN1}  The support of a nonzero effective anti-nef $\RR$-divisor $D$ on $X$ contains all exceptional prime divisors.
\end{Lemma}

\begin{proof} Let $S$ be the set of exceptional prime divisors which are in the support of $D$. Write $D=B+\sum_{i=1}^ra_iE_i$ where $B$ is an effective divisor which contains no exceptional prime divisors in its support and all $a_i\ge 0$. For all  $E_j$, we have that 
$$
0\ge (D\cdot E_j)=(B\cdot E_j)+\sum_{i\ne j}a_i(E_i\cdot E_j)+a_j(E_j^2),
$$
 and so
\begin{equation}\label{eqN70}
-a_j(E_j^2)\ge (B\cdot E_j)+\sum_{i\ne j}a_i(E_i\cdot E_j)\ge 0.
\end{equation}
If $B$ is nonzero, then there exists $E_{j}$ such that $(E_{j}\cdot B)>0$ and thus $a_{j}>0$ and so $E_j\in S$. If $B=0$ then there exists $E_j$ such that $(E_j\cdot D)<0$ since $D\ne 0$ and the intersection matrix $\left((E_i\cdot E_j)\right)$ is nonsingular. 
Thus $S$ is nonempty. If $E_{j'}\in S$ and $E_j$ is such that $(E_j\cdot E_{j'})>0$ then $E_j\in S$ by (\ref{eqN70}).
The exceptional fiber $\pi^{-1}(m_R)$ is connected as $R$ is normal and $\pi$ is birational (by \cite[Corollary III.11.4]{H}).
Thus $S$ is the set of all exceptional prime divisors of $X$.
\end{proof}

\begin{Lemma}\label{Lemma200} $X$ is the blowup of an $m_R$-primary ideal.
\end{Lemma}

\begin{proof}
Since the intersection matrix $\left((E_i\cdot E_j)\right)$ is negative definite,
 there exists an effective anti-ample $\QQ$-divisor $A$ on $X$ with exceptional support (by Proposition \ref{PropAmp}). Thus $-dA$ is very ample for some $d\in \ZZ_{>0}$. Let $I=\Gamma(X,\mathcal O_X(-dA))$. The ideal $I$ is $m_R$-primary since the support of $A$ is exceptional.
 The integral closure of $\sum_{n\ge 0}I^nt^n$ in $R[t]$ is 
 $$
 \sum_{n\ge 0}\overline{I^n}t^n=\sum_{n\ge 0}\Gamma(X,\mathcal O_X(-ndA))t^n.
 $$
  Since $R$ is excellent, $\sum_{n\ge 0}\overline{I^n}t^n$ is a finitely generated graded $R$-algebra.  Thus after replacing $d$ with a higher power of $d$ we may assume that $I^n=\overline {I^n}=\Gamma(X,\mathcal O_X(-ndA))$ for all $n\in \ZZ_{>0}$
(as follows from \cite[Proposition III.3.2 and Proposition III.3.3 on pages 158 and 159]{Bou}).
 
 Let $Y=\mbox{Proj}(\oplus_{n\ge 0}I^n)$, which is normal since $\oplus_{n\ge 0}I^n$ is integrally closed.  Since $\mathcal O_X(-dA)$ is generated by global sections we have that $I\mathcal O_X=\mathcal O_X(-dA)$. By the universal property of blowing up (\cite[Proposition II.7.14]{H}), there exists a unique $R$-morphism $\phi:X\rightarrow Y$ such that $\phi^*\mathcal O_Y(1)\cong \mathcal O_X(-dA)$. $\phi$ is a birational morphism which is an isomorphism away from the preimage of $m_R$. $\phi$ is of finite type since $X\rightarrow \mbox{Spec}(R)$ is. Since $(-A\cdot E)>0$ for all exceptional curves of $X$ we have that $\phi$ does not contract any curves of $X$ and thus $\phi$ is quasi-finite. Let $p\in X$ and $q=\phi(p)$. Let $A=\mathcal O_{Y,q}$ and $B=\mathcal O_{X,p}$.  The birational extension $A\rightarrow B$ satisfies $m_AB$ is $m_B$-primary since $\phi$ is quasi-finite.    Since $A$ is normal and excellent it is analytically irreducible by  \cite[Scholie IV.7.8.3(vii)]{EGAIV}. 
Thus by Zariski's main theorem \cite[(10.7) page 240]{Ab3} or \cite[Proposition 21.53]{AG}, we have that $A=B$ and so $\phi$ is an isomorphism and $X$ is the blowup of the $m_R$-primary ideal $I$.
 \end{proof}

\begin{Lemma}\label{Lemma00} Let $A$ be a universally Nagata domain and $I$ be an ideal in $A$. Let $Y=\mbox{Proj}(\bigoplus_{n\ge 0}I^n)$.
Then the graded ring $\bigoplus_{n\ge 0}\Gamma(Y,I^n\mathcal O_Y)$ is a finite $\bigoplus_{n\ge 0}I^n$-module and there exists $n_0\in \ZZ_{>0}$ such that  $\Gamma(Y,I^n\mathcal O_Y)=I^n$ for $n\ge n_0$.
\end{Lemma}

\begin{proof} This follows from the proof on the last two lines of page 122 through the first half of page 123 of \cite[Theorem II.5.19]{H}, along with the fact (observed in \cite[Remark 5.19.2]{H}) that the integral closure of a Nagata domain in its quotient field is a finite extension (by \cite[Proposition 31.B]{Mat}).
\end{proof}

\section{Riemann-Roch theorems for curves}\label{SecRR}
We summarize the famous Riemann-Roch theorems for curves. The following theorems are standard over algebraically closed fields. A reference where they are proven over an arbitrary field $k$ is \cite[Section 7.3]{Lin}. The results that we need are stated in \cite[Remark 7.3.33]{Lin}.

Let $E$ be an integral regular projective curve over a field $k$. For $\mathcal F$ a coherent sheaf on $E$ define
$h^i(\mathcal F)=\dim_kH^i(E,\mathcal F)$.

Let $D=\sum a_ip_i$ be a divisor on $E$, where $p_i$ are prime divisors on $E$ (closed points) and $a_i\in \ZZ$. We have an associated invertible sheaf $\mathcal O_X(D)$. Define
 $$
 \deg(D)=\deg(\mathcal O_E(D))=\sum a_i[\mathcal O_{E_i,p_i}/m_{p_i}:k].
 $$
The Riemann-Roch formula is
\begin{equation}\label{eq44}
\chi(\mathcal O_{E}(D)):=h^0(\mathcal O_{E}(D))-h^1(\mathcal O_{E}(D))=\mbox{deg}(D)+1-p_a(E)
\end{equation}
where $p_a(E)$ is the arithmetic genus of $E$. 

We further have Serre duality,
\begin{equation}\label{eq42}
H^1(E,\mathcal O_{E}(D))\cong H^0(E,\mathcal O_{E}(K-D))
\end{equation}
where $K=K_E$ is a canonical divisor on $E$. As a consequence, we have 
\begin{equation}\label{eq43}
\deg D>2p_a(E)-2=\deg(K)\mbox{ implies }H^1(E,\mathcal O_{E}(D))=0.
\end{equation}

We have the following  well known consequence of these formulas, which we record for future reference.

\begin{Lemma}\label{Lemma25} Let $E$ be an integral regular projective curve over a field $k$. Let $\{D_n\}_{n\ge 0}$ be an infinite sequence of  divisors on $E$ such that $\deg(D_n)$ is bounded from below and let $Z$ be a  divisor on $E$. Then there exists $s\in \ZZ_{>0}$ such that
$$
h^1(\mathcal O_{E}(D_n+Z))\le s\mbox{ for all }n\in \NN.
$$
\end{Lemma}

\begin{proof}
There exists an integer $c$  such that $\deg(D_n)\ge c$ for all $n$. Let $U$ be an effective divisor on $E$ of degree larger than $2p_a(E)-2+c$. 
By Serre duality (\ref{eq42}), 
$$
h^1(\mathcal O_{E}(D_n+Z))=h^0(\mathcal O_{E}(K-(D_n+Z))
$$
where $K$ is a cononical divisor on $E$. We have 
$$
\mbox{deg}(K-(Z+D_n))\le \deg(K-Z)-c.
$$
If $\mbox{deg}(K-Z)-c<0$, then certainly $h^0(\mathcal O_{E}(K-(D_n+Z))=0$. If 
$\mbox{deg}(K-Z)-c\ge 0$, then $h^1(\mathcal O_{E}(K-(D_n+Z)+U)=0$
by (\ref{eq43}) and so 
$$
\begin{array}{lll}
h^0(\mathcal O_{E}(K-(D_n+Z))&\le& h^0(\mathcal O_{E}(K-(D_n+Z)+U)\\
&=&\deg(K-(D_n+Z))+\deg(U)+1-p_a(E)\\
&\le& \deg(K-Z)-c+\mbox{deg}(U)+1-p_a(E).
\end{array}
$$
\end{proof}

If $\mathcal L$ is an invertible sheaf on $E$ then $\mathcal L\cong \mathcal O_E(D)$ for some divisor $D$ on $E$, and we may define $\deg(\mathcal L)=\deg(\mathcal O_X(D))=\deg(D)$.  

We will apply the above formulas in the case that $E$ is a prime exceptional divisor for a resolution of singularities $\pi:X\rightarrow\mbox{Spec}(R)$ as in Section \ref{SecRes}. We take $k=R/m_R$. We have that  $E$ is projective over $k=R/m_R$, and $E$ is a nonsingular (by assumption) integral curve.  Let $D$ be a divisor on $X$. Then
$\deg(\mathcal O_X(D)\otimes\mathcal O_E)=(D\cdot E)$.

\section{Zariski decomposition}\label{SecZD} In this section we present a relative form of the Zariski decomposition defined for projective surfaces over an algebraically closed field in \cite{Z}. Lemma \ref{Lemma10} in the case that $D$ is exceptional follows directly from \cite{Z} or \cite[Theorem 3.3]{BCK}.

We continue with our ongoing assumptions that  $R$ is  a two dimensional excellent normal local ring with quotient field $K$, maximal ideal $m_R$ and residue field $k=R/m_R$ and 
$\pi:X\rightarrow \mbox{Spec}(R)$ is a resolution of singularities  such that  
the  exceptional prime divisors  $E_1,\ldots,E_r$ are nonsingular.

The proof of the following lemma is a modification of the proof of \cite[Theorem 3.3]{BCK}.

\begin{Lemma}\label{Lemma10} Let $D$ be an effective $\RR$-divisor on $X$. Then there exist unique effective $\RR$-divisors $\Delta$ and $B$ on $X$ such that the following 1) and 2) hold. 
\begin{enumerate}
\item[1)] $\Delta=D+B$ is anti-nef and $B$ has exceptional support.
\item[2)] $(\Delta\cdot E)=0$ if $E$ is a component of $B$. 
\end{enumerate}
Further, 
\begin{enumerate}
\item[3)] $\Delta$ is the unique minimal effective anti-nef $\RR$-divisor such that $\Delta-D$ is effective with exceptional support.  
\item[4)] If $D$ is a $\QQ$-divisor then $\Delta$ and $B$ are $\QQ$-divisors. 
\end{enumerate}
\end{Lemma}

The decomposition $\Delta=D+B$ of the conclusions of Lemma \ref{Lemma10} is called the Zariski decomposition of $D$.

\begin{proof} 
For $x=(x_1,\ldots,x_r)\in \RR^r$, consider the inequalities 
\begin{equation}\label{eq10}
0\le x_i\mbox{ for }1\le i\le r
\end{equation}
and
\begin{equation}\label{eq11}
\left((D+\sum_{i=1}^rx_iE_i)\cdot E_j\right)\le 0\mbox{ for }1\le j\le r.
\end{equation}

Since the matrix $((E_i\cdot E_j))$ is negative definite and by Proposition \ref{PropAmp}, there exists an anti-ample, effective divisor $A=\sum_{i=1}^ra_iE_i$ on $X$.  Thus $a_i>0$ for all $i$ (by Lemma \ref{RemarkAN1}) and after possibly replacing $A$ with a positive multiple of $A$, $x=a=(a_1,\ldots,a_r)$ satisfies (\ref{eq10}) and (\ref{eq11}). 
Let
\begin{equation}\label{eq12}
S=\{x\in \RR^r\mid x_i\le a_i\mbox{ for  all $i$ and the $2r$ inequalities (\ref{eq10}) and (\ref{eq11}) are satisfied}\}.
\end{equation}
The set $S$ is nonempty and  compact. Thus there is at least one point in $S$  such that $\sum_{i=1}^rx_i$ is minimized on $S$. Let $b=(b_1,\ldots,b_r)$ be such a point. Let $B=b_1E_1+\cdots+b_rE_r$ and $\Delta=D+B$. Then $\Delta$ is an effective, anti-nef $\RR$-divisor and $B$ is an effective $\RR$-divisor with exceptional support. 
Let $E_j$ be a component of $B$. Since $b$ minimizes $\sum x_i$, $B-\epsilon E_j$ is effective and  $\Delta-\epsilon E_j$ is not anti-nef for all $\epsilon>0$  sufficiently small. But $\left((\Delta-\epsilon E_j)\cdot E_i\right)\le 0$ for all $i\ne j$ so we must have that  $\left((\Delta-\epsilon E_j)\cdot E_j\right)>0$ for all positive $\epsilon$ and thus $(\Delta\cdot E_j)=0$ since $\Delta$ is anti-nef. Thus the decomposition $\Delta=D+B$ satisfies 1) and 2).

For $b=(b_1,\ldots,b_r)$, $b'=(b_1',\ldots,b_r')\in \RR^r$, define 
$$
\mbox{min}(b,b')=(\mbox{min}(b_1,b_1'),\ldots,\mbox{min}(b_r,b_r')).
$$
If $b$ and $b'$ satisfy (\ref{eq10}) and (\ref{eq11}) then $\mbox{min}(b,b')$ also satisfies (\ref{eq10}) and (\ref{eq11}),
as we now show. For a fixed $j$, we may assume that $\mbox{min}(b_j,b_j')=b_j$ (after possibly interchanging $b$ and $b'$). Then 
since $(E_i\cdot E_j)\ge 0$ if $i\ne j$, we have that 
$$
((D+\sum_i\mbox{min}(b_i,b_i')E_i)\cdot E_j)\le ((D+\sum_ib_iE_i)\cdot E_j)\le 0.
$$

Suppose that $B=\sum b_iE_i$ and $B'=\sum b_i'E_i$ are effective $\RR$-divisors such that $\Delta=D+B$ and $\Delta'=D=B'$ satisfy both 1) and 2). We will show that $B=B'$ and so $\Delta=\Delta'$.  Let $\mbox{min}(B,B')=\sum_i\mbox{min}(b_i,b'_i)E_i$. There exist $x_i\ge 0$ such that  $\mbox{min}(B,B')=B-\sum_i x_iE_i$.  Since $D+\mbox{min}(B,B')$ is anti-nef, for each element $E_j$ of the support of $B$ we have 
$$
0\ge \left((D+\mbox{min}(B,B'))\cdot E_j\right)=\left((\Delta-\sum_i x_i E_i)\cdot E_j\right)=-\sum_i x_i(E_i\cdot E_j).
$$
Thus $\sum_i x_i(E_i\cdot E_j)\ge 0$ and so 
$$
\left((\sum_i x_iE_i)\cdot (\sum_j x_jE_j)\right)=\sum_i\sum_j x_ix_j(E_i\cdot E_j)\ge 0.
$$
Since the matrix $\left((E_i\cdot E_j)\right)$ is negative definite, we have that $x_i=0$ for all $i$. Thus $B=\mbox{min}(B,B')$. Similarily, $B'=\mbox{min}(B,B')$ and so $B=B'$.
Thus there is a  unique effective $\RR$-divisor $B$ with exceptional support such that $B$ and $\Delta=D+B$ satisfy 1) and 2).

We now show that $\Delta$ is the unique minimal effective and anti-nef $\RR$-divisor on $X$ such that $\Delta-D$ is effective with exceptional support. Let $U$ be an effective anti-nef $\RR$-divisor on $X$ such that $U-D$ is effective with exceptional support. Let $U'=D+\mbox{min}(\Delta-D,U-D)$. As shown earlier in the proof, $U'\ge D$ is effective and anti-nef. Write $U'-D=\sum u_iE_i$ and  $B=\Delta-D=\sum b_iE_i$. We have $\sum u_i\le \sum b_i\le \sum a_i$ so $U'-D\in S$ (defined in (\ref{eq12})). Since $\sum b_i$ is the minimum of $\sum x_i$ on $S$, we have that $u_i=b_i$ for all $i$ and so $U'=\Delta$. Thus $\Delta\le U$.

Now suppose that $D$ is an effective $\QQ$-divisor on $X$. Let $\Delta=D+B$ be the Zariski decomposition of $D$.
After possibly reindexing the $E_1,\ldots,E_r$, we may assume that the support of $B$ is $E_1\cup\cdots\cup E_s$ for some $s$ with $1\le s\le r$. Expand $D=F+\sum_{i=1}^rc_iE_i$ where $F$ is an effective $\QQ$-divisor whose support does not contain any prime exceptional divisor and $c_1,\ldots,c_r\in \QQ_{\ge 0}$. Then $\Delta=F+\sum_{i=1}^rd_iE_i$ with $c_i\le d_i$ for all $i$ and $d_i=c_i$ for $s+1\le i\le r$. Further, for $1\le j\le s$, we have
$0=(\Delta\cdot E_j)=\sum_{i=1}^sd_i(E_i\cdot E_j)+g_j$ where $g_j=(F\cdot E_j)+\sum_{i=s+1}^rc_i(E_i\cdot E_j)\in \QQ$. Since the $s\times s$ matrix $\left((E_i\cdot E_j)\right)_{1\le i,j\le s}$ is negative definite, and thus is nonsingular, we have that $d_1,\ldots,d_s\in \QQ$. Thus $\Delta$ and $B$ are $\QQ$-divisors.
\end{proof}

\begin{Remark}\label{Remarkv1} From 3) of the   conclusions of Lemma \ref{Lemma10},  we deduce that if $D_1\le D_2$ are effective 
$\RR$-divisors such that $D_2-D_1$ has  exceptional support and the respective anti-nef parts of their Zariski decompositions are $\Delta_1$ and $\Delta_2$, then $\Delta_1\le\Delta_2$.
\end{Remark}

\begin{Lemma}\label{LemmaV2}
Suppose that $D$ is an effective $\RR$-divisor on $X$ and $\Delta=D+B$ is the Zariski decomposition of $D$. Then for all $n\in \NN$, 
$$ 
\Gamma(X,\mathcal O_X(-\lceil nD\rceil))=\Gamma(X,\mathcal O_X(-\lceil n\Delta\rceil)).
$$
\end{Lemma}

\begin{proof} Suppose that $f\in \Gamma(X,\mathcal O_X(-\lceil n\Delta\rceil))$. Then $(f)-\lceil n\Delta\rceil\ge 0$. Writing  $n\Delta=\lceil n\Delta\rceil -G$ with $G\ge 0$, we have $-n\Delta=G-\lceil n\Delta\rceil$.
From 
$$
-nD=-n\Delta+nB=-\lceil n\Delta\rceil+(G+nB)
$$
and the fact that $G+nB\ge 0$, we have that $(f)-nD\ge 0$ so that $f\in \Gamma(X,\mathcal O_X(-\lceil nD\rceil))$.

Let $S$ be the set of prime divisors  in the support of $B$. Suppose that 
$$
f\in \Gamma(X,\mathcal O_X(-\lceil nD\rceil )).
$$
 Then $(f)-nD\ge 0$. Write $(f)-nD=A+C$ where $A$ and $C$ are effective $\RR$-divisors on $X$, no components of $A$ are in $S$ and all components of $C$ are in $S$. We have that $(f)-n\Delta=A+(C-nB)$. If $E\in S$ then 
$$
(E\cdot(A+(C-nB)))=(E\cdot((f)-n\Delta))=0
$$
which implies $(E\cdot (C-nB))=-(E\cdot A)\le 0$. The intersection matrix of the curves in $S$ is negative definite since it is so for the set of all exceptional curves, so $C-nB\ge 0$ (by \cite[Lemma 7.1]{Z}). Thus $(f)-n\Delta\ge 0$ which implies $(f)-\lceil n\Delta\rceil\ge 0$ since $(f)$ is an integral divisor.  Thus $f\in \Gamma(X,\mathcal O_X(-\lceil n\Delta\rceil))$.  
\end{proof}

\section{Nef divisors}\label{SecNef}  
In this section we extend to our relative situation $X\rightarrow \mbox{Spec}(R)$ some theorems proven by Zariski in \cite{Z} for projective  surfaces over an algebraically closed field. We stay as close as possible to Zariski's original proof, although some parts require modification. In \cite{LA}, and the references in that book, a theory of nef divisors on nonsingular projective varieties of arbitrary dimension  over an algebraically closed field of characteristic zero is derived.  Much of this theory can be extended to the relative situation, over $\mbox{Spec}(A)$, where the local ring $A$ is normal and essentially of finite type over an algebraically closed field of characteristic zero, or even of positive characteristic. 

We continue with our ongoing assumptions that 
$R$ is a two dimensional excellent normal local ring with quotient field $K$, maximal ideal $m_R$ and residue field $k$, and that $\pi:X\rightarrow \mbox{Spec}(R)$ is a resolution of singularities  such that  
the exceptional prime divisors $E_1,\ldots,E_r$ of $\pi$ are all nonsingular.  

\begin{Proposition}\label{Prop20} Let $\Delta$ be an effective anti-nef  divisor on $X$. For $n\ge 0$, let $B_n$ be the fixed component of $-n\Delta$.  Suppose that $E$  is a prime divisor which is in the support of the fixed component $B_n$ of $-n\Delta$ for infinitely many $n$. Then $E$ is exceptional for $\pi$ and $(\Delta\cdot E)=0$.
\end{Proposition}

\begin{proof}  By Lemma \ref{Lemma19}, $E$ is exceptional. We will assume that $(\Delta\cdot E)<0$ and derive a contradiction.
 Since $\Gamma(X,\mathcal O_X(-\Delta))\ne 0$ there exists an effective divisor $D$ on $X$ such that $D\sim -\Delta$.
Write $D = U+F_1+\cdots+F_s$ where $U$ is an effective divisor with no exceptional divisors in its support and $F_1=E,F_2,\ldots,F_s$ are prime exceptional divisors.  Let $\Delta_i=U+
F_1+\cdots+F_i$ for $0\le i\le s$. 

We have short exact sequences 
$$
0\rightarrow \mathcal O_X(nD-\Delta_0)\rightarrow \mathcal O_X(nD)\rightarrow \mathcal O_X(nD)\otimes\mathcal O_{\Delta_0}\rightarrow 0.
$$
There exists a  very ample effective divisor $H$ on $X$ which contains no exceptional prime divisors in its support and whose support
is disjoint from $\Delta_0$ by \cite[Theorem III.5.2]{H} since $\Delta_0$ intersects $\pi^{-1}(m_R)$ in only  a finite  number of closed points and so $\Delta_0$ is a closed subscheme of the affine scheme $X\setminus V(H)$ and thus $\Delta_0$ is an affine scheme.
We thus  have that $H^1(\Delta_0,\mathcal O_X(-nD)\otimes\mathcal O_{\Delta_0})=0$ for all $n$ and so 
\begin{equation}\label{eq35}
h^1(\mathcal O_X(nD))\le h^1(\mathcal O_X(nD-\Delta_0))
\end{equation}
for all $n\in\NN$.

For $i<s$ and $n\in\NN$, we have short exact sequences
$$
0\rightarrow \mathcal O_X(nD-\Delta_i-F_{i+1})\rightarrow \mathcal O_X(nD-\Delta_i)\rightarrow \mathcal O_X(nD-\Delta_i)\otimes \mathcal O_{F_{i+1}}\rightarrow 0.
$$
Thus 
$$
h^1(\mathcal O_X(nD-\Delta_i))\le h^1(\mathcal O_X(nD-\Delta_{i+1})+h^1(F_{i+1},O_X(nD-\Delta_i)\otimes \mathcal O_{F_{i+1}}).
$$
$(D\cdot F_{i+1})=(-\Delta\cdot F_{i+1})\ge 0$ implies that there exists $\sigma_i>0$ such that 
$h^1(F_{i+1},O_X(nD-\Delta_i)\otimes \mathcal O_{F_{i+1}})\le \sigma_i$ for all $n\in \NN$ by Lemma \ref{Lemma25}, so
\begin{equation}\label{eq30}
h^1(\mathcal O_X(nD-\Delta_i))\le h^1(\mathcal O_X(nD-\Delta_{i+1}))+\sigma_i
\end{equation}
for all $i\ge 0$ and $n\in \NN$.

Now consider the exact sequences
$$
0\rightarrow \mathcal O_X(nD-\Delta_0-F_1)\rightarrow \mathcal O_X(nD-\Delta_0)\rightarrow 
\mathcal O_X(nD-\Delta_0)\otimes \mathcal O_{F_1}\rightarrow 0
$$
for $n\in\NN$.
Since $(F_1\cdot D)=(F_1\cdot -\Delta)>0$ we have that $H^1(F_1,\mathcal O_X(nD-\Delta_i)\otimes \mathcal O_{F_1})=0$ for $n\gg0$ by (\ref{eq43}).  From the natural inclusion $\mathcal O_X(nD-\Delta_0)\rightarrow \mathcal O_X(nD)$ we deduce that $F_1$ is in the support of the fixed locus of $nD-\Delta_0$ if $F_1$ is in the support of the fixed locus of $-n\Delta$. Thus for $n$ such that $F_1$ is a component of the base locus $B_n$ of $-n\Delta$, the image of $H^0(X,\mathcal O_X(nD-\Delta_0))$ in $H^0(F_1,\mathcal O_X(nD-\Delta_i)\otimes \mathcal O_{F_1})$ is zero. Thus
$$
h^1(\mathcal O_X(nD-\Delta_0))=h^1(\mathcal O_X(nD-\Delta_0-F_1))-\chi(\mathcal O_{F_1}(nD-\Delta_0)\otimes\mathcal O_{F_1})
$$
so that by the Riemann  Roch theorem (\ref{eq44}),
\begin{equation}\label{eq31}
h^1(\mathcal O_X(nD-\Delta_0))=h^1(\mathcal O_X(nD-\Delta_0-F_1))+n(\Delta\cdot F_1)+(\Delta_0\cdot F_1)+p_a(F_1)-1.
\end{equation}

As explained before the statement of Lemma \ref{Lemma19}, there exists a positive integer $h$ such that for $n\gg0$, $F_1$ is a component of $B_n$ if $h\not|n$.

By (\ref{eq35}) and (\ref{eq30}), there exists a constant $c>0$ such that 
$$
h^1(\mathcal O_X(nD))\le h^1(\mathcal O_X((n-1)D))+c
$$
for all $n\in \ZZ_{>0}$ and for all $n\gg 0$ such that $h\not| n$ we have  by (\ref{eq35}), (\ref{eq30}) and (\ref{eq31}) that
$$
h^1(\mathcal O_X(nD))\le h^1(\mathcal O_X((n-1)D)+n(\Delta\cdot F_1)+c.
$$
Thus we have $h^1(\mathcal O_X(nD))<0$ for $n\gg0$ since we have assumed that  $(\Delta\cdot F_1)<0$. But this is impossible, giving a contradiction and so $(\Delta\cdot F_1)=0$.
\end{proof}

\begin{Proposition}\label{Prop50}
Let $\Gamma$ be an effective divisor on $X$ such that $-\Gamma$ has no fixed component. Then
\begin{enumerate}
\item[1)] $\mathcal O_X(-n\Gamma)$ is generated by global sections for all $n\gg0$.
\item[2)] There exists $s\in \ZZ_{>0}$ such that $h^1(X,\mathcal O_X(-n\Gamma))<s$ for all $n\in\NN$.
\end{enumerate}
\end{Proposition}

\begin{proof} The set of base points 
$$
\Omega=\{p\in X\mid \mathcal O_X(-\Gamma)_p\mbox{ is not generated by global sections}\}
$$
 of $\Gamma(X,\mathcal O_X(-\Gamma))$ is a finite set of closed points, which are necessarily contained in the exceptional fiber of $\pi$. 
Let $C\ge 0$ be an effective  divisor on $X$ such that $-C$ is very ample for $\pi$. There exists an integer $m>0$ such that there exists an effective divisor $H\sim -mC$ with no exceptional components in its support and such that $\Omega$ is disjoint from its support (by \cite[Theorem III.5.2]{H}).  After replacing $C$ with this multiple $mC$ we may assume that $H\sim -C$. Let $f\in K$, the quotient field of $R$, be such that $(f)-C=H$.  We may regard the effective divisor $H$ as a closed subscheme of $X$.

We have a short exact sequence
$$
0\rightarrow \mathcal O_X(C)\stackrel{f}{\rightarrow} \mathcal O_X\rightarrow \mathcal O_H\rightarrow 0
$$
and tensoring with $\mathcal O_X(-i\Gamma-jC)$ we have short exact sequences
\begin{equation}\label{eq60}
0\rightarrow \mathcal O_X(-i\Gamma-(j-1)C)\stackrel{f}{\rightarrow}\mathcal O_X(-i\Gamma-jC)\rightarrow \mathcal O_X(-i\Gamma-jC)\otimes\mathcal O_H\rightarrow 0.
\end{equation}
For $i,j\ge 0$, let $A_{i,j}$ be the natural image of $\Gamma(X,\mathcal  O_X(-i\Gamma-jC))$ in 
$$
\Gamma(H,\mathcal O_X(-i\Gamma-jC)\otimes \mathcal O_H),
$$
 upon taking global sections of (\ref{eq60}). Since the base points
of $\Gamma(X,\mathcal  O_X(-i\Gamma-jC))$ are a subset of $\Omega$ and so are disjoint from $H$, we have that, for all $i,j\ge 0$,  $A_{i,j}\mathcal O_{H,q}=\mathcal O_X(-i\Gamma-jC)_q$ for all $q\in H$. 

There exists $n\in \ZZ_{>0}$ such that there exists an effective  divisor $G$ on $X$ such that $G\sim -nC$, the support of $G$ contains no exceptional components of $\pi$ and $\mbox{sup}(H)\cap \mbox{sup}(G)\cap \mbox{sup}(\pi^{-1}(m_R))=\emptyset$ (by \cite[Theorem III.5.2]{H}). We may regard $G$ as a closed subscheme of $X$. Thus $H$ is a closed subscheme of the affine scheme $X\setminus G$ and so $H$ is affine, say $H=\mbox{Spec}(S)$. The restriction of $\pi$ to $H$ is determined by a ring homomorphism $R\rightarrow S$. Now $S=\Gamma(H,\mathcal O_H)$ is a finitely generated $R$-module  since $\pi$ is a projective morphism (by \cite[Corollary II.5.20]{H}).
As explained in \cite[Corollary II.5.5]{H}, since $S$ is Noetherian, the functor $M\rightarrow \tilde M$ gives an equivalence of categories between the category of finitely generated $S$-modules and the category of coherent $\mathcal O_{\mbox{Spec}(S)}$-modules, with inverse $\mathcal F\mapsto \Gamma(\mbox{Spec}(S),\mathcal F)$.

In particular, letting $B_{i,j}=\Gamma(H,\mathcal O_X(-i\Gamma-jC)\otimes \mathcal O_H)$ for $i,j\ge 0$, we have that 
$\mathcal O_X(-i\Gamma-jC)\otimes \mathcal O_H=\widetilde{B_{i,j}}$. We also have that $B_{i,j}$ is the tensor product over $S$ of $i$ copies of $B_{1,0}$ and $j$ copies of $B_{0,1}$ (\cite[Proposition II.5.2]{H}).

We have that the ring $A_{0,0}$ is a quotient of $\Gamma(X,\mathcal O_X)=R$ since $\pi$ is proper birational and $R$ is normal. Let $A_{0,0}[t_1,t_2]$ be a  polynomial ring over $A_{0,0}$, which is bigraded by specifying that $\deg(a)=(0,0)$ if $a\in A_{0,0}$, $\deg(t_1)=(1,0)$ and $\deg(t_2)=(0,1)$. Let $M$ be the  bigraded $A_{0,0}$-subalgebra $M:=\sum_{i,j\ge 0} A_{i,j}t_1^it_2^j$ of $A_{0,0}[t_1,t_2]$. Similarly, let $B$ be the bigraded $S$-subalgebra $B:=\bigoplus_{i,j\ge 0}B_{i,j}t_1^it_2^j$ of $S[t_1,t_2]$.

We have a natural inclusion of graded rings $M\rightarrow B$.

Since $H$ is disjoint from $\Omega$ we have that 
$$
A_{1,0}^iA_{0,1}^jS_q=A_{ij}S_q=\mathcal O_X(-i\Gamma-jA)\otimes\mathcal O_{H,q}=(B_{i,j})_q
$$
 for all $q\in H$ and $i,j\ge 0$. Thus 
 \begin{equation}\label{eq61}
 A_{1,0}^iA_{0,1}^jS=B_{i,j}\mbox{ for all $i,j\ge 0$}.
 \end{equation}
   Let $A$ be the bigraded $A_{0,0}$-subalgebra  $A:= A_{0,0}[A_{1,0}t_1,A_{0,1}t_2]$ of $M$. Now we have a natural surjection $A_{1,0}^iA_{0,1}^j\otimes_RS\rightarrow B_{i,j}$ for all $i,j\ge 0$ by (\ref{eq61}).
 Thus the natural homomorphism $A\otimes_RS\rightarrow B$ is surjective. Since $S$ is a finitely generated $R$-module, we have that $B$ is a finitely generated bigraded $A$-module. Since $A\subset M\subset B$ and $A$ is Noetherian, we have that $M$ is also a finitely generated $A$-module.

 By \cite[Lemma 4.3]{Z}, since $A$ is generated in bidegrees $(1,0)$ and $(0,1)$,
  and $M$ is a finitely generated bigraded $R$-module, there exists $N\in\ZZ_{>0}$ such that 
 \begin{equation}\label{eq62}
 A_{i,j}=A_{i,j-1}A_{0,1}\mbox{ whenever $j\ge N$ and $i\ge 0$ is arbitrary}
 \end{equation}
 and 
 \begin{equation}\label{eq63}
 A_{i,j}=A_{i-1,j}A_{1,0}\mbox{ whenever $i\ge N$ and $j\ge 0$ is arbitrary.}
 \end{equation} 
 Thus taking global sections in the short exact sequences (\ref{eq60}), and applying (\ref{eq63}), we have that if $i\ge N$ and $j\ge 0$, then
 \begin{equation}\label{eq64}
 \Gamma(X,\mathcal O_X(-i\Gamma-jC))=\Gamma(X,\mathcal O_X(-i\Gamma-(j-1)C))f+\Gamma(X,\mathcal O_X(-(i-1)\Gamma-jC))\Gamma(X,\mathcal O_X(-\Gamma)).
 \end{equation}
 Since $-C$ is ample, for fixed $i$, $\mathcal O_X(-i\Gamma-jC)$ is generated by global sections for all $j\gg 0$ (by \cite[Theorem II.5.17]{H}). Let $i$ be a fixed integer $\ge N$ and let $j>0$ be such that $\mathcal O_X(-i\Gamma-jC)$ is generated by global sections.  
 
 The only points  $q\in X$ where it is possible for  $\mathcal O_X(-i\Gamma-(j-1)C)_q$ to  not be generated by global sections are the points of $\Omega$. Suppose that $q\in\Omega$.  Thus $q$ is not in the support of $H=(f)-C$, and so $f=0$ is a local equation of $C$ at $q$ and $f\mathcal O_{X,q}=\mathcal O_X(-C)_q$. Further, since $q\in \Omega$, 
 $\Gamma(X,\mathcal O_X(-\Gamma))O_{X,q}\subset m_q\mathcal O_{X}(-\Gamma)$ where $m_q$ is the maximal ideal of $\mathcal O_{X,q}$, 
 equation (\ref{eq64}) and Nakayama's lemma show that 
 $$
 \begin{array}{lll}
 \mathcal O_X(-i\Gamma-jC)_q&=&\Gamma(X,\mathcal O_X(-i\Gamma-jC))\mathcal O_{X,q}\\
 &=&\Gamma(X,\mathcal O_X(-i\Gamma-(j-1)C))f\mathcal O_{X,q}\\
 &&+\Gamma(X,\mathcal O_X(-(i-1)\Gamma-jC)\mathcal O_X(-\gamma)m_q\\
 &=&\Gamma(X,\mathcal O_X(-i\Gamma-(j-1)C))\mathcal O_X(-C)_q.
 \end{array}
  $$ 
 Thus $\Gamma(X,\mathcal O_X(-i\Gamma-(j-1)C))\mathcal O_{X,q}= \mathcal O_X(-i\Gamma-(j-1)C)_q$, and since this is true for all $q\in \Omega$, 
 $\mathcal O_X(-i\Gamma-(j-1)C)$ is generated by global sections. 
 
 By descending induction on $j$, we obtain that $\mathcal O_X(-i\Gamma)$ is generated by global sections for all $i\ge N$.
 
 We now prove the second statement of the proposition. Let $g_0,\ldots,g_r\in \Gamma(X,\mathcal O_X(-N\Gamma))$ generate $\Gamma(X,\mathcal O_X(-N\Gamma))$ as an $R$-module. Then $g_0,\ldots,g_r$ induce a proper $R$-morphism
 $\phi:X\rightarrow \PP_R^r$ such that $\phi^*\mathcal O_{\PP_R^r}(1)\cong \mathcal O_X(-N\Gamma)$ (by \cite[Theorem II.7.1, Corollary II.4.8]{H}). In fact, $\phi$ is projective, by \cite[Proposition II.5.5 (v)]{EGAII} or \cite[Lemma 29.43.15, Tag 01W7]{SP} and    \cite[Lemma 29.43.16 (1), Tag 01W7]{SP}.  
 Let $Z$ be the image of $\phi$ in $\PP_R^r$ (which is closed since $\phi$ is 
  proper) and let $\mathcal O_Z(1)=\mathcal O_{\PP^r_R}(1)\otimes\mathcal O_Z$. Let $\overline\phi:X\rightarrow Z$ be the induced projective $R$-morphism.  By \cite[Corollary III.11.2]{H}, for $s\in \ZZ$, the  support of $R^1\overline\phi*\mathcal O_X(-s\Gamma)$ is contained in the finite set of closed points of $Z$ which are the images  of  curves contracted by $\overline\phi$ (the prime exceptional divisors $E$ of $\pi$ such that $(E\cdot-\Gamma)=0$). By \cite[Theorem II.5.19]{H}, $\Gamma(Z,R^1\overline\phi*\mathcal O_X(-s\Gamma))$  is a finitely generated $R$-module. Since it's support is the maximal ideal of $R$, the length of  $\Gamma(Z,R^1\overline\phi*\mathcal O_X(-s\Gamma))$ as an $R$-module is finite.  
  
  From  the Leray spectral sequence  we obtain exact sequences (\cite[Theorem 11.2]{Ro})    for $m\in \ZZ$,
  $$
  0\rightarrow H^1(Z,\overline \phi_*\mathcal O_X(-m\Gamma))\rightarrow H^1(X,\mathcal O_X(-m\Gamma))
  \rightarrow H^0(Z,R^1\overline\phi_*\mathcal O_X(-m\Gamma)).
  $$
   For $m\in\NN$, write $m=nN+s$ with $0\le s<N$. Then $\mathcal O_X(-m\Gamma)\cong\overline\phi^*\mathcal O_Z(n)\otimes\mathcal O_X(-s\Gamma)$.  Then by  the projection formula (\cite[Exercise III.8.3]{H}), we obtain exact sequences for $n,s\in \ZZ$
  \begin{equation}\label{eq65}
  \begin{array}{l}
  0\rightarrow H^1(Z,\mathcal O_Z(n)\otimes \overline\phi_*\mathcal O_X(-s\Gamma))
  \rightarrow H^1(X,\overline\phi^*\mathcal O_Z(n)\otimes\mathcal O_X(-s\Gamma))\\
  \rightarrow H^0(Z,(R^1\overline\phi_*\mathcal O_X(-s\Gamma))\otimes\mathcal O_Z(n)).
    \end{array}
  \end{equation}
  Let 
  $$
  s_1=\max\{\ell_R\Gamma(Z,R^1\overline\phi_*\mathcal O_X(-s\Gamma))\mid 0\le s<N\}.
  $$
 We have that $H^1(Z,\mathcal O_Z(n)\otimes \overline\phi_*\mathcal O_X(-s\Gamma))=0$ for all $0\le s<N$ and $n\gg 0$ (\cite[Theorem III.5.2]{H}). Let
  $$
  s_2=\max\{\ell_RH^1(Z,\mathcal O_Z(n) \otimes \overline\phi_*\mathcal O_X(-s\Gamma)) \mid 0\le s<N \mbox{ and }n\in \NN\}
  $$
   $s_2$ is finite by \cite[Proposition III.8.5, III.Theorem 8.8, Corollary III.11.2]{H}. By (\ref{eq65}), we have that 
   $\ell_RH^1(X,\mathcal O_X(-m\Gamma))\le s_1+s_2$ for all $m\in \NN$.
  \end{proof}

\begin{Proposition}\label{Prop21} Let $\Delta$ be an effective anti-nef divisor on $X$. For $n\ge 0$, let $B_n$ be the fixed component of $-n\Delta$. Then there  exists an effective exceptional divisor $G$ on $X$ such that $B_n\le G$ for all $n\in \ZZ_{>0}$.
\end{Proposition}

\begin{proof} To prove the proposition, it suffices to prove it for it for some positive multiple $d$ of $\Delta$, since for $n\in \NN$, writing $n=md+s$ with $0\le s<d$, we have $B_n\le B_{md}+B_s$. 

Write $-\Delta=\sum_{i=1}^t a_iF_i$.  
Let 
$$
M_i=\{n\in \NN\mid F_i\mbox{ is not a component of }B_n\}.
$$
$M_i$ is a numerical semigroup, so if $M_i$ is nonzero, there exists $h_i\in \ZZ_{>0}$ such that for $n\gg 0$, $n\in M_i$ if and only if $h_i$ divides $n$. Let
$$
\mathcal B(D)=\{F_i\mid F_i\mbox{ is a component of $B_n$ for infinitely many $n$}\}.
$$
By Proposition \ref{Prop20}, $F_i\in \mathcal B(D)$ implies $(F_i\cdot\Delta)=0$ and $F_i$ is exceptional for $\pi$. 
After possibly reindexing that $F_i$, we may assume that the support of $\mathcal B(D)$ is $\cup_{i=1}^sF_i$, for some $s\le t$.
We have that $M_i=0$ or $h_i>1$ for $1\le i\le s$. 
Thus the support of $B_n$ is $\cup_{i=1}^sF_i$ if $n\gg 0$ and $h_i\not | n$ for all $i$ such that $1\le i\le s$ and $M_i$ is non zero.



If we replace $\Delta$ with $n_0\Delta$ for some $n_0\gg 0$, we have that  the support of $B_1$ is $\mathcal B(D)$.
By Proposition \ref{Prop50},
 there exists $s_0\in \NN$ such that 
the effective divisor $\Gamma=\Delta+B_1$ satisfies 
the condition that  $h^1(\mathcal O_X(-n\Gamma))\le s_0$ for all $n\ge 1$ since $-\Gamma$ has no fixed component.

For a given $n\in \ZZ_{>0}$, consider the following conditions on a divisor $Z_n$.
\begin{enumerate}
\item[a)] $n\Gamma\ge Z_n\ge n\Delta$
\item[b)] $-Z_n$ has no fixed component
\item[c)] $h^1(\mathcal O_X(-Z_n))\le s_0$.
\end{enumerate}

Let $C_n$ be a minimal element in the set of divisors satisfying a), b) and c). Let $B_n'=C_n-n\Delta$. Then $nB_1\ge B_n'\ge B_n$ (since $-n\Delta=-n\Gamma+nB_1=-C_n+B_n'$ and $C_n\le n\Gamma$). Thus it suffices to show that the $B_n'$ are bounded from above. 

For $1\le i\le s$ we have short exact sequences 
$$
0\rightarrow O_X(-C_n)\rightarrow \mathcal O_X(-C_n+F_i)\rightarrow \mathcal O_X(-C_n+F_i)\otimes \mathcal O_{F_i}\rightarrow 0,
$$
giving  exact sequences 
$$
\begin{array}{l}
0\rightarrow H^0(X,\mathcal O_X(-C_n))\rightarrow H^0(X,\mathcal O_X(-C_n+F_i))\rightarrow H^0(F_i,\mathcal O_X(-C_n+F_i)\otimes \mathcal O_{F_i})\\
\rightarrow H^1(X,\mathcal O_X(-C_n))\rightarrow H^1(X,\mathcal O_X(-C_n+F_i))
\rightarrow H^1(F_i,\mathcal O_X(-C_n+F_i)\otimes \mathcal O_{F_i})\rightarrow 0.
\end{array}
$$
We will show that
\begin{equation}\label{eq70}
-(C_n\cdot F_i)\le \mbox{max}\{s_0-(F_i^2)-1+p_a(F_i),2p_a(F_i)-2-(F_i^2),0\}
\end{equation}
for all $n$ and $1\le i\le s$.

First assume that $F_i$ is not a component of $B_n'$. Then $(B_n'\cdot F_i)\ge 0$. Since $(F_i\cdot \Delta)=0$ by Proposition \ref{Prop20},  we have that $(C_n\cdot F_i)\ge 0$ and so (\ref{eq70}) holds. 

Now assume that $F_i$ is  a component of $B_n'$. We have that either  
\begin{equation}\label{eq32}
H^0(X,\mathcal O_X(-C_n+F_i))=H^0(X,\mathcal O_X(-C_n))
\end{equation}
 or 
 \begin{equation}\label{eq33}
 h^1(\mathcal O_X(-C_n+F_i))>s_0.
 \end{equation}
 
 If (\ref{eq32}) holds, then $h^0(\mathcal O_X(-C_n+F_i)\otimes O_{F_i})\le s_0$ . Thus 
 $$
 s_0\ge h^0(\mathcal O_X(-C_n+F_i)\otimes\mathcal O_{F_i})\ge ((-C_n+F_i)\cdot F_i)+1-p_a(F_i)
 $$
 by the Riemann-Roch formula (\ref{eq44}), and so (\ref{eq70}) holds. 
 
 Suppose that (\ref{eq33}) holds. Then $h^1(F_i,\mathcal O_X(-C_n+F_i)\otimes\mathcal O_{F_i})>0$, and so 
 $$
 ((-C_n+F_i)\cdot F_i)<2p_a(F_i)-2
 $$
  by (\ref{eq43}). Thus (\ref{eq70}) holds.

For $i$ with $1\le i\le s$, let $\sigma_i=\mbox{max}\{s_0-(F_i^2)-1+p_a(F_i),2p_a(F_i)-2-(F_i^2),0\}$.
Since $(F_i\cdot \Delta)=0$  for  $1\le i\le s$ by Proposition \ref{Prop20}, and by (\ref{eq70}), we have that
$$
(B_n'\cdot F_i)=((C_n-n\Delta)\cdot F_i)=(C_n\cdot F_i)\ge -\sigma_i.
$$
In particular, $\sigma_i\ge -(B_n'\cdot F_i)$.

Since the intersection matrix $((F_i\cdot F_j))$ for $1\le i,j\le s$ is negative definite, and thus is nonsingular, there exists a $\QQ$-divisor $\mathcal E=c_1F_1+\cdots +c_s F_s$  such that $(\mathcal E\cdot F_i)=-\sigma_i$ for $1\le i\le s$. Then
$$
((\mathcal E-B_n')\cdot F_i)=-\sigma_i-(B_n'\cdot F_i)\le 0
$$
for all $i$ implies $\mathcal E\ge B_n'$ by (\cite[Lemma 7.1]{Z}),   since the intersection matrix is negative definite. Thus the $B_n'$ are bounded from above.
\end{proof}

\begin{Corollary}\label{Lemma1} Let $\Delta$ be an effective anti-nef  $\QQ$-divisor on $X$. Let $B_n$ be the fixed component of $-\lceil n\Delta\rceil$; that is, the largest effective divisor on $X$ such that 
$$
\Gamma(X,\mathcal  O_X(-\lceil n\Delta\rceil ))=\Gamma(X,\mathcal O_X(-\lceil n\Delta\rceil -B_n)).
$$
 Then
\begin{enumerate}
\item[1)] The integral divisor $B_n$ has exceptional support for all $n\in \NN$ and
\item[2)] There exists an effective integral divisor $G$ with exceptional support such that $B_n\le G$ for all $n\in \ZZ_{>0}$.
\end{enumerate}
\end{Corollary}

\begin{proof} Statement 1) follows from Lemma \ref{Lemma19}. If $\Delta$ is an integral divisor then Statement 2) follows from Proposition \ref{Prop21}.
 
Now assume that $\Delta$ is a $\QQ$-divisor. Write $\Delta=\sum \frac{b_i}{d}F_i$ with $d\in \ZZ_{>0}$ and $b_i\in \NN$, where the $F_i$ are distinct prime divisors on $X$. Since $d\Delta$ is an integral divisor,  there exists an effective integral divisor $C$ with exceptional support such that $B_{nd}\le C$ for all $n\in \NN$. Let $n\in \NN$, and write $n=md-c$ with $m\in \NN$ and $0\le c<d$. Then 
$\mathcal O_X(-\lceil n\Delta\rceil)=\mathcal O_X(-md\Delta+\lfloor c\Delta\rfloor)$. Thus 
$B_n\le B_{md}+\lfloor c\Delta\rceil\le C+d\Delta$.
\end{proof}

\begin{Lemma}\label{Prop22} 
Let $\{D_n\}$ with $n\ge 0$ be an infinite sequence of divisors on $X$ and $Z$ be an effective divisor on $X$. If the sequence $h^1(\mathcal O_X(D_n))$ is bounded from above and if for each prime exceptional component $E$ of $Z$ $(D_n\cdot E)$ is bounded from below then $h^1(\mathcal O_X(D_n+Z))$ is bounded from above.
\end{Lemma}

\begin{proof} By induction on the number of components of $Z$, we may assume that $h^1(\mathcal O_X(D_n+Z-F))$ is bounded where $F$ is a prime component of $Z$. We have a short exact sequence 
$$
0\rightarrow \mathcal O_X(-F)\rightarrow \mathcal O_X\rightarrow \mathcal O_{F}\rightarrow 0,
$$
giving  exact sequences
$$
H^1(X,\mathcal O_X(D_n+Z-F))\rightarrow H^1(X,\mathcal O_X(D_n+Z))\rightarrow H^1(F,\mathcal O_{X}(D_n+Z)\otimes\mathcal O_{F}).
$$
If $F$ is exceptional, there exists $s\in \ZZ_{>0}$ such that $h^1(F,\mathcal O_X(D_n+Z)\otimes\mathcal O_{F})\le s$ for all $n\ge 0$ by Lemma \ref{Lemma25}, so $h^1(\mathcal O_X(D_n+Z))$ is bounded from above. If $F$ is not exceptional, then $F$ is affine and so $H^1(F,\mathcal O_X(D_m+Z))\otimes\mathcal O_{F})=0$ for all $m$, so again  $h^1(\mathcal O_X(D_n+Z))$ is bounded from above.
\end{proof}

\begin{Proposition}\label{Cor22} Let $\Delta$ be an effective anti-nef divisor on $X$. Then $h^1(\mathcal O_X(-n\Delta))$ is bounded for $n\in \NN$.
\end{Proposition}

\begin{proof} Let $C_n$ be the effective divisors of the proof of Proposition \ref{Prop21}, so that $B_n'=C_n-n\Delta$ are effective divisors and there exists an effective divisor $G$ with exceptional support such that $B_n'\le G$ for all $n\in \NN$. Since $-\Delta$ is nef, we have that $(-C_n\cdot E)$ is bounded from below for each prime exceptional component $E$ of $G$. Further, we have (by the proof of Proposition \ref{Prop21}) that $h^1(\mathcal O_X(-C_n))\le s_0$ for all $n\in \NN$. For each effective divisor $Z\le G$, Proposition \ref{Prop22} gives us an upper bound for $h^1(\mathcal O_X(-C_n+Z))$ over $n\in \NN$. The maximum of these bounds is an upper bound for $h^1(\mathcal O_X(-n\Delta))$ over $n\in \NN$.
\end{proof}

\begin{Corollary}\label{Cor23} Let $\Delta$ be an effective anti-nef divisor on $X$ and $\mathcal F$ be a coherent sheaf on $X$. Then $h^1(\mathcal O_X(-n\Delta)\otimes \mathcal F)$ is bounded for $n\in \NN$.
\end{Corollary}

\begin{proof} There exists an effective anti-ample  divisor $A$ on $X$ with exceptional support by Proposition \ref{PropAmp}. There exists $n_0\in \ZZ_{>0}$ such that $\mathcal F\otimes \mathcal O(-n_0A)$ is generated by global sections, so there is a surjection $\mathcal O_X^s\rightarrow \mathcal F\otimes \mathcal O_X(-n_0A)$ for some $s$, giving a short exact sequence of coherent sheaves
$$
0\rightarrow \mathcal K\rightarrow \mathcal O_X(n_0A)^s\rightarrow \mathcal F\rightarrow 0
$$
and surjections 
$$
H^1(X,\mathcal O_X(-n\Delta+n_0A))^s\rightarrow H^1(X,\mathcal O_X(-n\Delta)\otimes\mathcal F).
$$
Thus $h^1(\mathcal O_X(-n\Delta)\otimes\mathcal F)$ is bounded above for $n\in \NN$ since $-\Delta$ is nef, and by Lemma \ref{Prop22} and Proposition \ref{Cor22}.
\end{proof}

\section{Asymptotic properties of divisors on a resolution of singularities }\label{SecAS}
We continue with the notation introduced in the introduction and in Section \ref{SecRes}. 
We assume  that 
$R$ is a two dimensional excellent normal local ring with quotient field $K$, maximal ideal $m_R$ and residue field $k$, and that $\pi:X\rightarrow \mbox{Spec}(R)$ is a resolution of singularities  such that  
the exceptional prime divisors $E_1,\ldots,E_r$ of $\pi$ are all nonsingular.  

As explained in the introduction,
If $F$ is prime divisor on $X$ and $\alpha\in \RR_{\ge 0}$, then there is a  valuation ideal 
$I(\nu_F)_{\alpha}=\{f\in R\mod \nu_F(f)\ge\alpha\}$ of $R$, where $\nu_F$ is the valuation of the discrete (rank 1) valuation ring $\mathcal O_{X,F}$.

\begin{Proposition}\label{Prop1} Suppose that $\Delta_1\subset\Delta_2$ are effective anti-nef $\QQ$-divisors on $X$
 such that $\Delta_1\ne \Delta_2$. Then there exists $n_0\in \ZZ_{>0}$ such that $\Gamma(X,\mathcal O_X(-\lceil n\Delta_2\rceil))\ne \Gamma(X,\mathcal O_X(-\lceil n\Delta_1\rceil))$ for all $n\ge n_0$.
\end{Proposition}

\begin{proof} Write $\Delta_1=\sum a_iF_i$ and $\Delta_2=\sum b_iF_i$ where the $F_i$ are distinct prime divisors on $X$. We have $b_i\ge a_i$ for all $i$ and $b_j>a_j$ for some $j$. If $F_j$ is not exceptional then certainly $\Gamma(X,\mathcal O_X(-\lceil n\Delta_2\rceil ))\ne \Gamma(X,\mathcal O_X(-\lceil n\Delta_1\rceil ))$ for $n$ sufficiently large by  Lemma \ref{Lemma19}.

Now suppose that $F_j$ is exceptional. By 2) of Lemma \ref{Lemma1}, there exists an effective exceptional divisor $H=\sum c_iF_i$ such that the fixed component $B_n$ of $\Gamma(X,\mathcal O_X(-\lceil n\Delta_1\rceil))$ satisfies $B_n\le H$ for all $n\in \NN$. Observe that $g\in \Gamma(X,\mathcal O_X(-\lceil n\Delta_2\rceil))$ implies $\nu_j(g)\ge \lceil n b_j\rceil$. By definition of $B_n$,  for $n\in \ZZ_{>0}$, there exists $f_n\in \Gamma(X,\mathcal O_X(-\lceil n\Delta_1\rceil))$ such that $(f_n)-\lceil n\Delta_1\rceil=A_n+B_n$ where $F_j$ is not a component of the effective divisor $A_n$. Thus $\nu_j(f_n)=\lceil na_j\rceil +\delta$ with $\delta\le c_j$. We have that $n>\frac{c_j+1}{b_j-a_j}$ implies $\lceil na_j\rceil+\delta<\lceil nb_j\rceil$. Thus $\nu_j(f_n)<\lceil nb_j\rceil$ so that $f_n\not\in \Gamma(X,\mathcal O_X(-\lceil n\Delta_2\rceil))$.
\end{proof}

\begin{Corollary}\label{Prop0} Suppose that $\Delta_1\subset\Delta_2$ are effective anti-nef $\QQ$-divisors on $X$. Then the following are equivalent.
\begin{enumerate}
\item[1)] $\Gamma(X,\mathcal O_X(-\lceil n\Delta_1\rceil))=\Gamma(X,\mathcal O_X(-\lceil n\Delta_2\rceil))$ for infinitely many $n\in \ZZ_{>0}$.
\item[2)] $\Gamma(X,\mathcal O_X(-\lceil n\Delta_1\rceil))=\Gamma(X,\mathcal O_X(-\lceil n\Delta_2\rceil))$ for all $n\gg 0$
\item[3)] $\Delta_1=\Delta_2$.
\end{enumerate}
\end{Corollary}

\begin{proof} Proposition \ref{Prop1} proves the essential implication 1) implies 3). The directions   3) implies 2) and 2) implies 1) are immediate.
\end{proof}

\begin{Proposition}\label{Prop2} Let $\Delta=\sum_{i=1}^s a_iF_i$ be an effective anti-nef $\QQ$-divisor on $X$ and 
$E$ be a prime exceptional divisor on $X$. Then $E=F_j$ for some $j$ with $a_j>0$. 
The following are equivalent
\begin{enumerate}
\item[1)] There exists $n\in \ZZ_{>0}$ such that 
$$
I(n\Delta)=\cap_{i=1}^s I(\nu_{F_i})_{na_i}\ne \cap_{i\ne j}I(\nu_{F_i})_{na_i}.
$$
\item[2)] There exists $n_0\in \ZZ_{>0}$ such that 
$$
I(n\Delta)=\cap_{i=1}^s I(\nu_{F_i})_{na_i}\ne \cap_{i\ne j}I(\nu_{F_i})_{na_i}.
$$
for all $n\ge n_0$.
\item[3)] $(\Delta\cdot F_j)<0$.
\end{enumerate}
\end{Proposition}

\begin{proof}  It follows from Lemma \ref{RemarkAN1} that $E=F_j$ for some $j$ with $a_j>0$.

Let $D_1=\sum_{i\ne j}a_iF_i$, so that $D_1\le \Delta$.
Let $\Delta_1=D_1+B_1$ be the Zariski decomposition of $D_1$. We have that $\Delta_1\le \Delta$ by Remark \ref{Remarkv1}, and so
$0\le \Delta-\Delta_1=a_jF_j-B_1$ so that $0\le B_1\le a_jF_j$. Thus $\Delta_1=\Delta-\lambda F_j$ with $0\le \lambda\le a_j$.

If $\Delta_1\ne \Delta$ then $\lambda >0$, and so
\begin{equation}\label{eq2}
(F_j\cdot\Delta)=(F_j\cdot\Delta_1)+\lambda(F_j^2)<0.
\end{equation}

If $\Delta_1=\Delta$ then $B_1=a_jF_j$. Since  $a_j>0$, we have that  
\begin{equation}\label{eq3}
0=(\Delta_1\cdot F_j)=(\Delta\cdot F_j).
\end{equation}
by 2) of Lemma \ref{Lemma10}.

Suppose that 1) holds. Then $\Delta_1\ne \Delta$ so that $(F_j\cdot\Delta)<0$ by (\ref{eq2}), so that 1) implies 3) holds. Certainly 2) implies 1) is true, so we are reduced to proving 3) implies 2). Now 3) implies $\Delta_1\ne\Delta$ by (\ref{eq2}) and (\ref{eq3}).
If 2) doesn't hold then there exist infinitely many $n\in \ZZ_{>0}$ such that $\Gamma(X,\mathcal O_X(-\lceil n\Delta\rceil))=\Gamma(X,\mathcal O_X(-\lceil n\Delta_1\rceil))$ so that $\Delta_1=\Delta_2$ by Corollary \ref{Prop0}, giving a contradiction.
\end{proof}

\begin{Corollary}\label{Cor1} 
Let $\Delta=\sum_{i=1}^s a_iF_i$ be an effective anti-nef $\QQ$-divisor on $X$ and $E$ be a prime exceptional divisor on $X$ so that $E=F_j$ for some $j$ with $a_j>0$. 
The following are equivalent
\begin{enumerate}
\item[1)] 
$I(n\Delta)=\cap_{i=1}^r I(\nu_{F_i})_{na_i}= \cap_{i\ne j}I(\nu_{F_i})_{na_i}$
for all $n\in \ZZ_{>0}$.
\item[2)] 
 $(\Delta\cdot F_j)=0$.
\end{enumerate}
\end{Corollary}

\begin{Corollary}\label{Cor3} Suppose that $\Delta$  is an effective  anti-nef $\QQ$-divisor on $X$. Then the following are equivalent.
\begin{enumerate}
\item[1)] There exists $n$ such that $m_R\in \mbox{Ass}(R/I(n\Delta))$.
\item[2)] There exists $n_0$ such that $m_R\in \mbox{Ass}(R/I(n\Delta))$ for all $n\ge n_0$.
\item[3)] There exists a prime exceptional divisor $E$ for $\pi$ such that  $(\Delta\cdot E)<0$.
\end{enumerate}
\end{Corollary}

\begin{proof} Write $\Delta=\sum_{i=1}^sa_iF_i$, so that $I(n\Delta)=\cap_{i=1}^sI(\nu_{F_i})_{na_i}$.
For a fixed $n$, we  have that $m_R\in\mbox{Ass}(R/\cap_{i=1}^sI(\nu_{F_i})_{na_i})$ if and only if 
$$
\cap_{i=1}^sI(\nu_{F_i})_{na_i}\ne \cap_{F_i\mbox{ which are not exceptional}}I(\nu_{F_i})_{na_i}
$$
which occurs if and only if there exists $j$ such that $F_j$ is exceptional and 
$$
\cap_{i=1}^sI(\nu_{F_i})_{na_i}\ne \cap_{i\ne j}I(\nu_{F_i})_{na_i}
$$
Thus by Proposition \ref{Prop2}, the three conditions of the corollary are equivalent.
\end{proof}

Let $\Delta=\sum_{i=1}^sa_iF_i$ be an effective and anti-nef $\QQ$-divisor on $X$. By Lemma \ref{RemarkAN1}, all prime exeptional divisors $E_1,\ldots,E_r$ are in the support of $\Delta$. After permuting the $F_i$, we may assume that $F_i=E_i$ and $a_i>0$ for $1\le i\le r$. We have that
$$
R[\Delta]:= \bigoplus_{n\ge 0}\Gamma(X,\mathcal O_X(-\lceil n\Delta \rceil))=\bigoplus_{n\ge 0}\cap_{i=1}^sI(\nu_{F_i})_{na_i}.
$$
Let $P_j=\bigoplus_{n\ge 0}\Gamma(X,\mathcal O_X(-\lceil n\Delta\rceil-E_j))$ for $1\le j\le r$. 
We have that 
\begin{equation}\label{eq81}
\Gamma(X,\mathcal O_X(-E_j))=\{f\in R\mid \nu_{E_j}(f)>0\}=m_R
\end{equation}
for $1\le j\le r$.
for all $j$. Suppose that $f\in \Gamma(X,\mathcal O_X(-\lceil m\Delta\rceil))$ and $g\in \Gamma(X,\mathcal O_X(-\lceil n\Delta\rceil))$ are such that $fg\in \Gamma(X,\mathcal O_X(-\lceil (m+n)\Delta\rceil-E_j)$. Then 
$$
\nu_{E_j}(f)+\nu_{E_j}(g)=\nu_{E_j}(fg)\ge (m+n)a_j+1
$$
 implies $\nu_{E_j}(f)\ge ma_j+1$ or $\nu_{E_j}(g)\ge na_j+1$ so that $f\in \Gamma(X,\mathcal O_X(-\lceil m\Delta\rceil -E_j)$ or $g\in \Gamma(X,\mathcal O_X(-\lceil n\Delta\rceil -E_j)$. Thus $P_j$ is a prime ideal in $R[\Delta]$.

If $f\in m_R$, then $\nu_{E_j}(f)\ge 1$ for $1\le j\le r$ so that 
\begin{equation}\label{eq80}
m_RR[\Delta]\subset P_j.
\end{equation}

We have  exact sequences
$$
0\rightarrow P_j\rightarrow R[\Delta]\rightarrow \bigoplus_{n\ge 0}\Gamma(E_j,\mathcal O_X(-\lceil n\Delta\rceil)\otimes\mathcal O_{E_j}).
$$

\begin{Remark}\label{Remarkdim0}  Suppose that $\Delta$ is an effective anti-nef $\QQ$-divisor on $X$. Then $\dim R[\Delta]/P_j=0$ if and only if $R[\Delta]/P_j=R/m_R$.
\end{Remark}

\begin{proof} Suppose that for some $m>0$ there exists $f\in \Gamma(X,\mathcal O_X(-\lceil m\Delta\rceil))$ such that it's class $\overline f$ in $\Gamma(X,\mathcal O_X(-\lceil m\Delta\rceil))/\Gamma(X,\mathcal O_X(-\lceil m\Delta\rceil-E_j))$ is nonzero. Then 
$$
\overline ft^m\in \sum_{n=0}^{\infty}\Gamma(X,\mathcal O_X(-\lceil n\Delta\rceil))/\Gamma(X,\mathcal O_X(-\lceil n\Delta\rceil-E_j))t^n=R[\Delta]/P_j
$$
 is nonzero. The element $\overline ft^m$ is not a unit since it is homogeneous of positive degree and it is not nilpotent since $R[\Delta]/P_j$ is an integral domain.
 Thus  $\dim R[\Delta]/P_j>0$. Thus by  (\ref{eq81}), $\dim R[\Delta]/P_j=0$ implies $R[\Delta]/P_j=R/m_R$. 
 \end{proof}

\begin{Proposition}\label{Prop80}  Suppose that $\Delta$ is an effective anti-nef $\QQ$-divisor on $X$.  Then  
$$
\sqrt{m_RR[\Delta]}=\cap_{i=1}^rP_i.
$$
\end{Proposition}

\begin{proof} We have that  $\sqrt{m_RR[\Delta]}\subset \cap_{i=1}^rP_i$ by (\ref{eq80}).

Let $h\in \cap_{i=1}^rP_i$. We will show that $h^n\in m_RR[\Delta]$ for some $n\in \ZZ_{>0}$, which will establish the proposition.
We may assume that $h$ is homogeneous, so that 
$$
h\in \cap_{i=1}^r\Gamma(X,\mathcal O_X(-\lceil a\Delta\rceil-E_i))=
\Gamma(X,\mathcal O_X(-\lceil a \Delta\rceil-E_1-\cdots-E_r)
$$
for some $a\in \NN$. We must show that $h^n\in m_R\Gamma(X,\mathcal O_X(-\lceil an\Delta\rceil))$ for some $n\in \ZZ_{>0}$.

First suppose that $a=0$. We have that $\Gamma(X,\mathcal O_X(-E_1-\cdots-E_r)=m_R$ so we already have that 
$h\in m_R\Gamma(X,\mathcal O_X)=m_R$.

Now suppose that $a>0$. After replacing $\Delta$ with a positive multiple of $\Delta$ and $h$ with a power of $h$ we may assume that $\Delta$ is an integral divisor and $h\in \Gamma(X,\mathcal O_X(-\Delta-\sum_{i=1}^rE_i))$.
By Lemma \ref{Lemma200}, there exists an $m_R$-primary ideal $I$ in $R$ such that $X$ is the blowup of $I$, so that
$X=\mbox{Proj}(\bigoplus_{n\ge 0}I^n)$ and $I\mathcal O_X=\mathcal O_X(-C)$ is very ample, where $C$ is an effective divisor whose support is the union of all exceptional prime divisors $E_1,\ldots, E_r$. 
The graded ring $\bigoplus_{n\ge 0}\Gamma(X,I^n\mathcal O_X)$ is a finite $\bigoplus_{n\ge 0}I^n$-module and there exists $n_0\in \ZZ_{>0}$ such that the $R$-ideal $\Gamma(X,I^n\mathcal O_X)=I^n$ for $n\ge n_0$ by Lemma \ref{Lemma00}. Since $R$ and $X$ are normal, $\Gamma(X,I^n\mathcal O_X)=\overline{I^n}$ for all $n\ge 0$.

After possibly replacing $I$ with a positive power of $I$ we may assume that $\Gamma(X,I^n\mathcal O_X)=I^n$ for all $n\in \NN$ and that there exists an effective divisor $H\sim -C$ on $X$ with no exceptional prime divisors in its support. Let $f\in \Gamma(X,\mathcal O_X(-C))=I$ be such that $(f)-C=H$. We have a short exact sequence
$$
0\rightarrow \mathcal O_X(C)\stackrel{f}{\rightarrow}\mathcal O_X\rightarrow \mathcal O_H\rightarrow 0.
$$
There exists $\alpha\in \QQ_{>0}$ such that $F:=\sum_{i=1}^rE_i-\alpha C\ge 0$. There exists $e\in \ZZ_{>0}$ such that $e\alpha C$ is an integral divisor and so $eF$ is an integral divisor.  Thus for $n\in \ZZ_{>0}$, we have that 
$$
\begin{array}{lll}
h^{n2e}&\in& \Gamma(X,\mathcal O_X(-n2e\Delta-n2e(\sum_{i=1}^rE_i))=\Gamma(X,\mathcal O_X(-n2e\Delta-n2e\alpha C-n2eF)\\
&\subset& \Gamma(X,\mathcal O_X(-n2e\Delta-n2e\alpha C))=\Gamma(X,\mathcal O_X(-n2e(\Delta+\frac{\alpha}{2}C)-n2e\frac{\alpha}{2}C)).
\end{array}
$$
Now the effective integral divisor $2e(\Delta+\frac{\alpha}{2}C)$ is anti-ample by Proposition \ref{PropAmp}, since $\Delta$ is anti-nef. Thus there exists $n_0\in \ZZ_{>0}$ such that $\mathcal O_X(-n2e(\Delta+\frac{\alpha}{2}C))$ is generated by global sections for all $n\ge n_0$. Let $\Gamma=n_02e(\Delta+\frac{\alpha}{2}C)$.
By the argument of the proof of Proposition \ref{Prop50}, applying (\ref{eq62}), there exists $N>0$ such that
$$
\Gamma(X,\mathcal O_X(-i\Gamma-jC))=\Gamma(X,\mathcal O_X(-i\Gamma-(j-1)C)\Gamma(X,\mathcal O_X(-C))+f\Gamma(X,\mathcal O_X(-i\Gamma-(j-1)C)
$$
whenever $j\ge N$ and $i\ge 0$. Since $f\in \Gamma(X,\mathcal O_X(-C))$, we have that
$$
\begin{array}{lll}
\Gamma(X,\mathcal O_X(-i\Gamma-jC))&=&\Gamma(X,\mathcal O_X(-i\Gamma-(j-1)C))\Gamma(X,\mathcal O_X(-C))\\ 
&=&I\Gamma(X,\mathcal O_X(-i\Gamma-(j-1)C))\subset I\Gamma(X,\mathcal O_X(-in_02e\Delta)).
\end{array}
$$
Thus 
$$
h^{nn_02e}\in \Gamma(X,\mathcal O_X(-n\Gamma-\frac{nn_02e\alpha}{2}C))\subset I\Gamma(X,\mathcal O_X(-nn_02e\Delta))
\subset m_R\Gamma(X,\mathcal O_X(-nn_02e\Delta))$$
whenever $n$ is so large that $n\ge \frac{N}{n_0e\alpha}$.
\end{proof}

\begin{Corollary}\label{Remark82}  Suppose that $\Delta$ is an effective anti-nef $\QQ$-divisor on $X$. Then 
$$
\dim R[\Delta]/m_RR[\Delta]=0
$$
 if and only if the image of $\Gamma(X,\mathcal O_X(-\lceil n\Delta\rceil))$ in $\Gamma(E_j,\mathcal O_X(-\lceil n\Delta\rceil)\otimes\mathcal O_{E_j})$ is zero for $1\le j\le r$ and for all $n>0$. 
 \end{Corollary}
 

\begin{proof} By Proposition \ref{Prop80}, we have that $\dim R[\Delta]/m_RR[\Delta]=0$ if and only if $\dim R[\Delta]/P_j=0$ for all $j$, and this second conditions holds if and only if $R[\Delta]/P_j=R/m_R$ for all $j$ by Remark \ref{Remarkdim0}.
\end{proof}

\begin{Proposition}\label{Prop3} Suppose that $\Delta$ is an effective anti-nef $\QQ$-divisor on $X$ and $E_j$ is a prime  exceptional divisor for $\pi:X\rightarrow\mbox{Spec}(R)$. Then 
\begin{enumerate}
\item[1)] $\dim R[\Delta]/P_j=2$ if  $(\Delta\cdot E_j)<0$.
\item[2)] $\dim R[\Delta]/P_j\le 1$ if $(\Delta\cdot E_j)=0$.
\end{enumerate}
\end{Proposition}


\begin{proof} Suppose that $(\Delta\cdot E_j)<0$. We have short exact sequences
$$
0\rightarrow \mathcal O_X(-\lceil n\Delta\rceil-E_j)\rightarrow \mathcal O_X(-\lceil n\Delta\rceil )\rightarrow \mathcal O_X(-\lceil n\Delta\rceil)\otimes\mathcal O_{E_j}\rightarrow 0.
$$
Taking global sections we have short exact sequences
$$
\begin{array}{l}
0\rightarrow \Gamma(X,\mathcal O_X(-\lceil n\Delta\rceil-E_j))
\rightarrow 
\Gamma(X,\mathcal O_X(-\lceil n\Delta\rceil))\\\rightarrow \Gamma(E_j,O_X(-\lceil n\Delta\rceil)\otimes\mathcal O_{E_j})
\rightarrow H^1(X,\mathcal O_X(-\lceil n\Delta\rceil-E_j)).
\end{array}
$$
There exists $d\in \ZZ_{>0}$ such that $d\Delta$ is an integral divisor. By Corollary \ref{Cor23}, applied to $d\Delta$ and the coherent sheaves $\mathcal O_X(-\lceil s\Delta\rceil-E_j)$ for $0\le s<d$, we have that 
$$
h^1(X,\mathcal O_X(-\lceil n\Delta\rceil -E_j))
$$
is bounded for positive $n$. 
Since $(-\Delta\cdot E_j)>0$, we have (by the Riemann Roch theorem (\ref{eq44})) that there exists $c'>0$ such that 
$$
h^0(\mathcal O_X(-\lceil n\Delta\rceil )\otimes\mathcal O_{E_j})>c'n
$$
for $n\gg 0$. Thus there exists $c>0$ such that the image $A_n:=\mbox{Im}(\Gamma(X,\mathcal O_X(-\lceil n\Delta\rceil))$ in $B_n:=\Gamma(E_j,\mathcal O_X(-\lceil n\Delta\rceil )\otimes\mathcal O_{E_j})$ satisfies 
\begin{equation}\label{eqF1}
\ell_R(\Gamma(X,\mathcal O_X(-\lceil n\Delta\rceil))/\Gamma(X,\mathcal O_X(-\lceil n\Delta\rceil -E_j))=  \ell_R(A_n)=\dim_kA_n\ge cn
\end{equation}
for $n\gg 0$.

Let $A=\oplus_{n\ge 0}A_n$. We have that  $B_{0}$ is a finite field extension of $k=R/m_R=A_0$.
Now $\mathcal O_X(-d\Delta)\otimes\mathcal O_{E_j}$ is ample on the projective curve $E_j$, so
there exists $e\in \ZZ_{>0}$ which is divisible by $d$ such that $\mathcal O_X(-e\Delta)\otimes\mathcal O_{E_j}$ is very ample and $\overline B=\oplus_{m\ge 0}B_{me}$ is a finitely generated  $B_{0}$-algebra  which is generated  by its terms of the lowest positive degree $me$ (\cite[Theorem II.5.19 and Exercise II.5.14]{H}). Thus $\overline B$ is the coordinate ring of a projective embedding of the curve $E_j$ in a projective space over $B_0$,  determined by a $B_0$-basis of $\Gamma(E_j,\mathcal O_X(-e\Delta)\otimes\mathcal O_{E_j})$. Thus $\overline B$ has dimension two. Let $\overline A=\bigoplus_{m\ge 0}A_{me}$.



By (\ref{eqF1}), for $n\gg 0$, there exists $F\in A_{ne}$ such that $0\ne F$. The ring $\overline B_{(F)}$ of elements of degree zero in the localization $\overline B_F$ is such that $\mbox{Spec}(\overline B_{(F)})$ is the affine variety
 $E_j\setminus V(F)$, with maximal ideals in $\overline B_{(F)}$ corresponding to height one homogeneous prime ideals in 
 $\mbox{Proj}(\overline B)$ which do not contain $F$ (by \cite[Proposition II.2.5]{H}).
 Thus there exists a homogeneous height one prime ideal $Q=\oplus_{n> 0}Q_{ne}$ in $\overline B$ which does not contain $F$.

Let $P=\overline A\cap Q$, where $P=\oplus_{n> 0}P_{ne}$ with $P_{ne}=Q_{ne}\cap A_{ne}$. $\dim \overline B/Q=1$ implies that there exists $d\in \ZZ_{>0}$ such that $\dim_k(B_{ne}/Q_{ne})<d$ for all $n$  (by \cite[Theorem 4.1.3]{BH}).
Thus by (\ref{eqF1}) we have that $P\ne 0$. $P$ is not the graded maximal ideal $\oplus_{n\ge 0}A_{ne}$ of $\overline A$ since $F\not\in P$.

We have constructed a chain of distinct homogeneous prime ideals
$0\subset P\subset \bigoplus_{n>0}A_{ne}$ in $\overline A$ and thus 
$\overline A$ has dimension $\ge 2$. The extension $\overline A\rightarrow A$ is integral so $\dim A\ge 2$ by the going up theorem (\cite[Theorem 5.11]{AM}).
We have that $m_R\Gamma(X,\mathcal O_X(-\lceil n\Delta\rceil))\subset \Gamma(X,\mathcal O_X(-\lceil n\Delta\rceil-E_j))$ for all $n\ge 0$ by (\ref{eq80}). We thus have a surjection $R[\Delta]/m_RR[\Delta]\rightarrow A$ and so 
$\dim A\le \dim R[\Delta]/m_RR[\Delta]$. But $\dim R[\Delta]/m_RR[\Delta]\le 2$ by  \cite[Lemma 3.6]{CS}, so that  $\dim A=2$.

Now suppose that $(\Delta\cdot E_j)=0$. Let $B_n=\Gamma(E_j,\mathcal O_X(-\lceil n\Delta\rceil)\otimes \mathcal O_{E_j})$ and  $A_n$ be the natural image of $\Gamma(X,\mathcal O_X(-\lceil n\Delta\rceil))$ in $B_n$. We have that $A_0\cong R/m_R=k$ and $B_0$ is a finite field extension of $k$. Let  $A=\sum_{n\ge 0}A_nt^n$  where $t$ is an indeterminate. We have that $A\cong R[\Delta]/P_j$.

By the Riemann-Roch Theorem (\ref{eq44}) and Lemma \ref{Lemma25}, there exists $d>0$ such that $\dim_k(B_n)<d$ for all $n\in \NN$.

For $a\in \ZZ_{>0}$, define ${}_aA=\sum_{n\ge 0}{}_aA_nt^n$ to be the graded subring of $A$ defined by 
${}_aA=k[A_1t,A_2t^2,\ldots,A_at^a]$. The ring ${}_aA$ is a finitely generated graded $k$-algebra. For fixed $a$, there exists $e\in \ZZ_{>0}$ such that ${}_aA^{(e)}=\sum_{n\ge 0}{}_aA_{en}t^{en}$ is generated in degree $e$ 
(as follows from \cite[Proposition III.3.2 and Proposition III.3.3 on pages 158 and 159]{Bou}). Since ${}_aA$ is a finitely generated ${}_aA^{(e)}$-module, we have that $\dim {}_aA=\dim {}_aA^{(e)}$. Since $\dim_k {}_aA^{(e)}_n<d$ for all $n\in \NN$, we have that $\dim {}_aA\le 1$ for all $a\in \ZZ_{>0}$ by \cite[Theorem 4.1.3]{BH}. Suppose that $Q_0\subset Q_1\subset \cdots \subset Q_s$ is a chain of distinct prime ideals in $A$. Since $\cup_{a\ge 0}({}_aA)=A$,
 for all $a\gg0$,
$Q_0\cap {}_aA\subset Q_1\cap {}_aA\subset \cdots \subset Q_s\cap {}_aA$ is a chain of distinct prime ideals in $A$. Thus $\dim A\le 1$.
\end{proof}

\begin{Corollary}\label{Cor4} Suppose that $\Delta$ is an effective anti-nef $\QQ$-divisor on $X$. Then 
$$
\dim R[\Delta]/m_RR[\Delta]=2
$$
 if and only if there exists an exceptional prime divisor $E$ of $\pi$  such that  $(\Delta\cdot E)<0$
\end{Corollary}

\begin{proof} This follows from Propositions \ref{Prop80} and \ref{Prop3}.
\end{proof}



\section{Analytic spread of divisorial filtrations}\label{SecAS2}
Theorem \ref{Cor2} is a generalization to (not necessarily Noetherian) divisorial filtrations on a two dimensional normal local ring of a theorem of McAdam, for filtrations of powers of ideals,    in \cite{McA}   and  \cite[Theorem 5.4.6]{HS}. We recall the exact statement of McAdam's  theorem in  Theorem \ref{TheoremC4} of the introduction. The concept of a divisiorial filtration $\mathcal I(D)=\{I(nD)\}$ is defined in the introduction. 

\begin{Theorem}\label{Cor2} Let $R$ be a two dimensional normal excellent local ring. The following are equivalent for a $\QQ$-divisorial filtration $\mathcal I(D)$ on $R$.
\begin{enumerate}
\item[1)]
The analytic spread $\ell(\mathcal I(D))=\dim R[D]/m_RR[D]=2$.
\item[2)] $m_R\in \mbox{Ass}(R/I(nD))$ for some $n$.
\item[3)] There exists $n_0\in \ZZ_{>0}$ such that $m_R\in \mbox{Ass}(R/I(nD))$ for all $n\ge n_0$.
\end{enumerate}
\end{Theorem}

\begin{proof} Let $\pi:X\rightarrow \mbox{Spec}(R)$ be a resolution of singularities such $D=\sum_{i=1}^s a_iF_i$ for some prime divisors $F_i$ on $X$ and the exceptional divisors $E_1,\ldots, E_r$ of $\pi$ are nonsingular. Let
$\Delta=D+B$ be the Zariski decomposition of $D$ on $X$, so that $\mathcal I(D)=\mathcal I(\Delta)$ and $R[D]=R[\Delta]$ (by Lemma \ref{LemmaV2}). Then this theorem follows from Corollary \ref{Cor4} and \ref{Cor3}.
\end{proof}




\begin{Corollary} Let $R$ be a two dimensional normal excellent local ring and  $\mathcal I(D)$ be a $\QQ$-divisorial filtration on $R$.  Then $\dim R[D]/m_RR[D]\le 1$ if and only if there exist  height one prime ideals $Q_1,\ldots, Q_s$ in $R$ and $b_1,\ldots,b_r\in \QQ_{>0}$ such that  $I(nD)= Q_1^{(\lceil nb_1\rceil)}\cap \cdots \cap Q_s^{(\lceil nb_s\rceil)}$ for all $n\in\NN$. 
\end{Corollary}

\begin{proof}  We have that $I(nD)= Q_1^{(\lceil nb_1\rceil)}\cap \cdots \cap Q_s^{(\lceil nb_s\rceil)}$ for all $n\in\NN$ if and only if  $m_R\not\in \mbox{Ass}(R/I(nD))$ for all $n$ which holds if and only if $\dim R[D]/m_RR[D]\le 1$ by Corollary \ref{Cor2}.
\end{proof}

\begin{Example}\label{Ex1}  There exists  a $\QQ$-divisorial filtration $\mathcal I(D)$ on a two dimensional normal excellent local ring $R$ such that  the analytic spread $\ell(\mathcal I(D))=0$ and height 
$$
{\rm ht}(\mathcal I(D))={\rm ht}(I(D))=1,
$$
 giving an example where ${\rm ht}(\mathcal I(D))>\ell(\mathcal I(D))$.
The Rees algebra of the example is a Non Noetherian symbolic algebra  $R[D]=\sum_{n\ge0}Q_1^{(n)}\cap Q_2^{(n)}\cap Q_3^{(n)}$ where $Q_1, Q_2, Q_3$ are height one prime ideals in $R$.
 \end{Example}

\begin{proof} Let $k$ be an algebraically closed field and $F$ be an irreducible cubic form in the polynomial ring $k[x,y,z]$ such that $E=\mbox{Proj}(k[x,y,z]/(F))$ is an elliptic curve. Let $R=k[[x,y,z]]/(F)$, a complete, normal excellent local ring of dimension two with maximal ideal $m_R=(x,y,z)$. Let $\pi:X\rightarrow \mbox{Spec}(R)$ be the blow up of the maximal ideal $m_R$ of $R$. $X$ is nonsingular with $\pi^{-1}(m_R)\cong E$,
$m_R\mathcal O_X=\mathcal O_X(-E)$, $\mathcal O_X(-E)\otimes \mathcal O_E\cong \mathcal O_E(1)$ and $(E^2)=-3$. We have that $\mathcal O_X(-E)\otimes\mathcal O_E\cong\mathcal O_E(q_1+q_2+q_3)$ for some closed points $q_1,q_2,q_3\in E$. Let $p_1,p_2,p_3 \in E$ be distinct closed points on $E$ such that the degree 0 invertible sheaf $\mathcal L=\mathcal O_E(q_1+q_2+q_3-p_1-p_2-p_3)$ has infinite order in  $\mbox{Pic}^0(X)$. Then $h^0(\mathcal L^n)=0$ for all $n\in \ZZ$.
In each regular local ring $\mathcal O_{X,p_i}$, let $u_i,v_i$ be a regular system of parameters such that $u_i=0$ is a local equation of $E$ at $p_i$. Let $F_i$ be the Zariski closure of $v_i=0$ in $X$, which is an integral curve. Let $\pi(F_i)=Q_i\in \mbox{Spec}(R)$. $R/Q_i$ is Henselian since it is complete, so by \cite[Theorem 4.2 page 32]{Mil}, we have that $E$ intersects the integral curve $F_i$ only at the point $p_i$. $F_i$ intersects $E$ transversally at $p_i$ so that $(E\cdot F_i)=1$.
Let $D=F_1+F_2+F_3$. The Zariski decomposition of $D$ is $\Delta=D+E$. We have that $\mathcal O_X(-n\Delta)\otimes\mathcal O_E\cong \mathcal L^n$ for all $n$. Thus $\Gamma(X,\mathcal O_X(-n\Delta-E))=\Gamma(X,\mathcal O_X(-n\Delta))$ for all $n\in \ZZ_{>0}$, and so by Proposition \ref{Prop80}, 
$$
R[\Delta]/\sqrt{m_RR[\Delta]}=\bigoplus_{n\ge 0}\Gamma(X,\mathcal O_X(-n\Delta))/\Gamma(X,\mathcal O_X(-n\Delta-E))
=R/m_R=k.
$$
 Thus
$$
\dim R[\Delta]/m_RR[\Delta]=\dim R[\Delta]/\sqrt{m_RR[\Delta]}=0.
$$
Since $0=\ell(\mathcal I(D))<1=\mbox{ht}(\mathcal I(D))$, we have that $R[D]$ is Non Noetherian (by \cite[Proposition 3.7]{CPS}).
\end{proof}

\section{The Hilbert function of $R[D]/m_RR[D]$}\label{SecHilb}

\begin{Theorem}\label{HilbThm} Suppose that $R$ is a two dimensional normal excellent local ring and $\mathcal I(D)$ is a $\QQ$-divisorial filtration on $R$.
Then there exist a nonnegative rational number $\alpha$ and a bounded function $\sigma:\NN\rightarrow\QQ$ such that 
$$
\ell_R(I(nD)/m_RI(nD))=\ell_R((R[D]/m_RR[D])_n)=n\alpha+\sigma(n)
$$
for $n\in \NN$.
The constant $\alpha$ is positive if and only if $\dim(R[D]/m_RR[D])=2$.
\end{Theorem}

The function $\sigma$ is bounded from both above and below.  The proof gives an explicit calculation of the constant $\alpha$ in terms of the intersection theory of a suitable resolution of singularities in equation (\ref{Hilb9}).
The constant $\alpha$ is a nonnegative integer if $\Delta$ is an integral divisor in the  Zariski decomposition $D=\Delta+B$.

\begin{proof} There exists a resolution of singularities  $\pi:X\rightarrow \mbox{Spec}(R)$  such that $D$ is an effective $\QQ$-divisor on $X$, $m_R\mathcal O_X$ is invertible and  the  prime exceptional divisors $E_1,\ldots, E_r$ of $X$ are all nonsingular. Let $G$ be the effective exceptional divisor such that $m_R\mathcal O_X=\mathcal O_X(-G)$.
Let $\Delta=D+B$ be the Zariski decomposition of $D$ on $X$. There exists $d\in \ZZ_{>0}$ such that $d\Delta$ is an integral divisor.

Suppose that the ideal $m_R$ is generated by $f_1,\ldots,f_b$. We have an induced short exact sequence of coherent sheaves on $X$
$$
0\rightarrow \mathcal K\rightarrow \mathcal O_X^b\rightarrow m_R\mathcal O_X\rightarrow 0.
$$
Tensoring with $\mathcal O_X(-\lceil n\Delta\rceil)$ and taking global sections, we have short exact sequences
$$
0\rightarrow m_R\Gamma(X,\mathcal O_X(-\lceil  n\Delta\rceil)\rightarrow \Gamma(X,\mathcal O_X(-\lceil n\Delta\rceil-G))\rightarrow H^1(X,\mathcal K\otimes \mathcal O_X(-\lceil n\Delta\rceil)).
$$

Thus there exists $c_1\in \ZZ_{>0}$ such that
\begin{equation}\label{Hilb8}
\ell_R(\Gamma(X,\mathcal O_X(-\lceil n\Delta\rceil-G))/m_R\Gamma(X,\mathcal O_X(-\lceil  n\Delta\rceil))\le c_1
\end{equation}
for all $n\in \NN$ by Corollary \ref{Cor23}, applied to the effective anti-nef divisor $d\Delta$ and the coherent sheaves $\mathcal F=\mathcal K\otimes \mathcal O_X(-\lceil s\Delta\rceil)$ for  $0\le s<d$. From the short exact sequences
$$
0\rightarrow \mathcal O_X(-\lceil n\Delta\rceil-G)\rightarrow \mathcal O_X(-\lceil n\Delta\rceil) \rightarrow \mathcal O_X(-\lceil n\Delta\rceil)\otimes\mathcal O_G\rightarrow 0
$$
we have inclusions for $n\in \NN$
$$
\Gamma(X,\mathcal O_X(-\lceil n\Delta\rceil))/\Gamma(X,\mathcal O_X(-\lceil n\Delta\rceil-G))\rightarrow \Gamma(G,\mathcal O_X(-\lceil n\Delta\rceil)\otimes\mathcal O_G)
$$
and by Corollary \ref{Cor23}, there exists $c_2\in \ZZ_{>0}$ such that
\begin{equation}\label{Hilb7}
|\ell_R(\Gamma(G,\mathcal O_X(-\lceil n\Delta\rceil)\otimes\mathcal O_G))
-\ell_R(\Gamma(X,\mathcal O)X(-\lceil n\Delta\rceil))/\Gamma(X,\mathcal O_X(-\lceil n\Delta\rceil)))|
\le c_2.
\end{equation}
 We are reduced to computing
$h^0(\mathcal O_X(-\lceil n\Delta\rceil)\otimes\mathcal O_G)$ for $n\in \NN$.
Write $G=\sum_{i=1}^r a_iE_i$ with $a_i\in \ZZ_{>0}$.


Let $e=\sum_{i=1}^ra_i$. There exists a function $\tau:\{1,\ldots,e\}\rightarrow \{1,\ldots,r\}$ such that letting $C_1=E_{\tau(1)}$ and $C_{j+1}=C_j+E_{\tau(j+1)}$ for $1\le j<e$, we have that $C_e=G$.
We have short exact sequences
\begin{equation}\label{Hilb5}
0\rightarrow \mathcal O_X(-C_j)\otimes\mathcal O_{E_{\tau(j+1)}}\rightarrow \mathcal O_{C_{j+1}}\rightarrow \mathcal O_{C_j}\rightarrow 0
\end{equation}
for $1\le j<e$.
The cohomology groups $h^1(\mathcal O_X(-\lceil n\Delta\rceil -mE_j)\otimes\mathcal O_{E_{\tau(j+1)}})$
 are bounded for $1\le j<e$ and $n\in \NN$ by Lemma \ref{Lemma25}. Let 
$$
f=\max\{h^1(\mathcal O_X(-\lceil n\Delta\rceil -mE_j)\otimes\mathcal O_{E_{\tau(j+1)}})     \mid 1\le j<e\mbox{ and }n\in \NN\}.
$$
Tensoring the sequences (\ref{Hilb5}) with $\mathcal O_X(-\lceil n\Delta \rceil)$ and taking cohomology, we find that 
\begin{equation}\label{Hilb2}
|h^0(\mathcal O_X(-\lceil n\Delta\rceil)\otimes \mathcal O_{C_{j+1}})
-h^0(\mathcal O_X(-\lceil n\Delta\rceil)\otimes \mathcal O_{C_j})-h^0(\mathcal O_X(-\lceil n\Delta\rceil -C_j)\otimes\mathcal O_{E_{\tau(j+1)}})|\le f
\end{equation}
for $1\le j<e$ and $n\in \NN$.
Setting $C_0=0$, we have that there exists $\lambda\in \ZZ_{>0}$ such that
\begin{equation}\label{Hilb4}
|h^0(X,\mathcal O_X(-\lceil n\Delta\rceil)\otimes\mathcal O_G)-\sum_{i=0}^{e-1}h^0(X,\mathcal O_X(-\lceil n\Delta\rceil-C_i)\otimes\mathcal O_{E_{\tau(i+1)}})|<\lambda
\end{equation}
for all $n\in \NN$.
Writing $n=md+s$ with $0\le s<d$, we have
$$
h^0(\mathcal O_X(-\lceil n\Delta\rceil-C_j)\otimes \mathcal O_{E_{\tau(j+1)}})
=h^0(\mathcal O_X(-md\Delta-\lceil s\Delta\rceil-C_j)\otimes\mathcal O_{E_{\tau(j+1)}}).
$$
By Lemma \ref{Lemma25} and the Riemann-Roch theorem (\ref{eq44}), there exists $g\in \ZZ_{>0}$ such that
\begin{equation}\label{Hilb3}
|h^0(\mathcal O_X(-md\Delta-\lceil s\Delta\rceil -C_j)\otimes\mathcal O_{E_{\tau(j+1)}})-md(-\Delta\cdot E_{\tau(j+1)})| \le g
\end{equation}
for $1\le j <e$ and $m\in \NN$. Thus the theorem holds with 
\begin{equation}\label{Hilb9}
\alpha=(-\Delta\cdot G).
\end{equation}
\end{proof}

\end{document}